\documentclass [11pt]{amsart}
\usepackage[top=1.2in, bottom=1.2in, left=1in, right=1in]{geometry}
\usepackage{xcolor}
\usepackage{amssymb,amsmath,amsthm,amsfonts,mathrsfs,amscd}
\usepackage{enumerate}
\usepackage{setspace}
\usepackage{empheq}
\usepackage[all]{xy}
\usepackage{verbatim}
\usepackage[normalem]{ulem}
\usepackage{todonotes}
\usepackage{mathtools}
\usepackage[bookmarks,colorlinks,breaklinks]{hyperref}  
\hypersetup{linkcolor=blue,citecolor=blue,filecolor=blue,urlcolor=blue} 
\usepackage{graphicx}
\usepackage{float}
\usepackage{tikz-cd}

\usepackage{color}
\usepackage{xcolor}

\theoremstyle{plain}
\newtheorem{thm}{Theorem}[section]
\newtheorem{cor}[thm]{Corollary}
\newtheorem{prop}[thm]{Proposition}
\newtheorem{lem}[thm]{Lemma}

\theoremstyle{definition}
\newtheorem{defn}[thm]{Definition}

\newtheorem{aDD^+m}[thm]{ADD^+endum}

\theoremstyle{remark}
\newtheorem{rmk}[thm]{Remark}

\title{Tropicalizing polynomial strata}
\author{Yan Mary He}
\address{Department of Mathematics\\
	University of Oklahoma\\
	Norman, OK 73019}
\email{he@ou.edu}

\author{Chenxi Wu}
\address{Department of Mathematics, University of Wisconsin-Madison, Madison,
WI 53703}
\email{cwu367@wisc.edu}

\date{\today}

\begin{document}

\begin{abstract}
Let ${\rm Poly}_D(\vec\mu^*)$ be the ramification stratum in the parameter space of degree $D \geq 2$ complex polynomials consisting of polynomials with ramification profile $\vec\mu^*$. In this paper, we introduce a space $\mathcal{T}_D(\vec\mu^*)$ of framed decorated polynomial trees and identify it with the dynamical tropicalization of ${\rm Poly}_D(\vec\mu^*)$.

A non-trivial point is that the tropicalization of ${\rm Poly}_D(\vec\mu^*)$ is not available a priori, as the stratum does not come with a previously known proper toroidal compactification. We resolve this issue by identifying ${\rm Poly}_D(\vec\mu^*)$ with a rigidified framed Hurwitz space. Using the twisted admissible cover compactification, together with additional rigidifying data, we construct a proper toroidal compactification and hence an associated Berkovich skeleton. We prove that $\mathcal{T}_D(\vec\mu^*)$ is isomorphic to this skeleton.

We further show that the projectivized tree space $\mathbb P\mathcal{T}_D(\vec\mu^*)$ compactifies ${\rm Poly}_D(\vec\mu^*)$. Finally, we compare this compactification with the DeMarco--McMullen compactification of polynomial moduli by projectivized space of polynomial trees.
\end{abstract}

\maketitle

\section{Introduction}
A central theme in complex dynamics is that degenerating families of maps often acquire limiting combinatorial structures. For polynomials, the work of DeMarco--McMullen \cite{DM08} shows that divergent sequences in the moduli space can be compactified by projectivized polynomial trees. These trees provide a natural dynamical boundary; however, they are not, by themselves, obtained as the Berkovich skeleton of a proper toroidal compactification of the moduli space. On the other hand, compactifications obtained directly from polynomial coefficients are algebraic but do not reflect the intrinsic dynamics of degeneration. This motivates the study of tropicalizations of polynomial parameter spaces through compactifications that are both toroidal in the algebro-geometric sense and compatible with the dynamics of degenerating polynomials.

Let ${\rm Poly}_D:=\left\{f(z)=a_Dz^D+a_{D-1}z^{D-1}+\cdots+a_0 \;:\; a_D\neq 0\right\}$
be the space of all complex polynomials of degree $D\ge 2$. Then ${\rm Poly}_D$ is a complex manifold of dimension $D+1$, and it admits a stratification by ramification profile
$${\rm Poly}_D=\bigsqcup_{\vec\mu^*}{\rm Poly}_D(\vec\mu^*),$$
where the disjoint union runs over all ramification profile $\vec\mu^*=((D),\mu^2, \dots,\mu^r)$. Here $(D)$ records the totally ramified branch value at $\infty$, the partitions $\mu^2,\ldots,\mu^r$ of $D$ record the prescribed non-simple finite ramification profiles, and the remaining finite branch values are simple. 
Each stratum ${\rm Poly}_D(\vec\mu^*)$ is a smooth locally closed subvariety of ${\rm Poly}_D$, and 
$\dim_{\mathbb{C}} {\rm Poly}_D(\vec\mu^*)=r+s+1$, where
$s$ is the number of simple finite branch values.

In this paper, for each ramification profile $\vec\mu^{*}$, we introduce a space $\mathcal{T}_D(\vec\mu^*)$ of {\it framed decorated polynomial trees} and identify it with the {\it tropicalization} of ${\rm Poly}_D(\vec\mu^{*})$.
We note that the tropicalization of the stratum ${\rm Poly}_D(\vec\mu^{*})$ is non-trivial as ${\rm Poly}_D(\vec\mu^{*})$ does not admit a previously known proper toroidal compactification compatible with the dynamics.

To resolve this issue, we introduce
the {\it rigidified framed Hurwitz space} $\mathcal{H}_{\mathrm{rig},0,\mathbb P^1,D}(\vec\mu^*)$.  
We prove that $\mathcal{H}_{\mathrm{rig},0,\mathbb P^1,D}(\vec\mu^*)$ admits a proper toroidal compactification, and hence a canonical {\it Berkovich skeleton} $\Sigma\!\left(\overline{\mathcal{H}}^{{\rm an}}_{\mathrm{rig},0,\mathbb P^1,D}(\vec\mu^*)\right)$ as such given by Thuillier \cite{Thuillier07} and by Abramovich--Caporaso--Payne \cite{ACP15}. Moreover, we introduce the {\it rigidified framed tropical Hurwitz space} $\overline{\mathcal{H}}^{{\rm trop}}_{\mathrm{rig},0,\mathbb P^1,D}(\vec\mu^*)$ and 
prove that it is isomorphic to  
the Berkovic skeleton; see Section \ref{sec_rfhur} and Theorem \ref{thm_rig_framed_CMR16}.
The crucial point is that the stratum ${\rm Poly}_D(\vec\mu^{*})$ is naturally isomorphic to $\mathcal{H}_{\mathrm{rig},0,\mathbb P^1,D}(\vec\mu^*)$; see Remark \ref{rmk_polyD}. Therefore ${\rm Poly}_D(\vec\mu^{*})$ admits a canonical tropicalization.

Our first main theorem states that 
the Berkovich skeleton $\Sigma\!\left(\overline{\mathcal{H}}^{{\rm an}}_{\mathrm{rig},0,\mathbb P^1,D}(\vec\mu^*)\right)$, i.e., the tropicalization of ${\rm Poly}_D(\vec\mu^{*})$, is  isomorphic to the space of framed decorated polynomial trees of profile $\vec\mu^{*}$. 
\begin{thm}\label{thm:main-1}
For each polynomial ramification profile $\vec\mu^*$, there is a canonical isomorphism of extended cone complexes
$$
\Phi_{\vec\mu^*}\colon \overline{\mathcal{T}}_D(\vec\mu^*) \xrightarrow{\sim} 
\overline{\Sigma}\!\left(\overline{\mathcal{H}}^{{\rm an}}_{\mathrm{rig},0,\mathbb P^1,D}(\vec\mu^*)\right).
$$
In particular, $\Phi_{\vec\mu^*}$ is a homeomorphism.
\end{thm}

Therefore the tropicalization of each polynomial stratum is described explicitly in terms of decorated metric trees, and the tropical geometry of the polynomial parameter space is related directly to the combinatorial and dynamical structure encoded by these trees.

\medskip

We equip $\mathcal T_D(\vec\mu^*)$ with the cone-complex topology
coming from its description as a finite extended cone complex. 
On the other hand, we also define a {\it geometric topology} on $\mathcal T_D(\vec\mu^*)$ based on convergence of sequences of framed decorated polynomial trees.
We prove that these two topologies agree; see
Proposition \ref{prop:2-top-agree}. 

We next prove that the {\it projectivized} space of framed decorated polynomial tree provides a compactification of ${\rm Poly}_D(\vec\mu^*)$. The space $\mathcal{T}_D(\vec\mu^*)$ carries a natural scaling action of
$\mathbb R_{>0}$, obtained by multiplying all edge length and boundary
length parameters by the same positive constant. We define the
projectivized tree space by
$$
\mathbb{P}\mathcal{T}_D(\vec\mu^*)
:=
\bigl(\mathcal{T}_D(\vec\mu^*)\setminus\{0\}\bigr)/\mathbb R_{>0}.
$$
Thus $\mathbb{P}\mathcal{T}_D(\vec\mu^*)$ records only the relative rates of degeneration,
or equivalently the projective class of the associated toroidal valuation
vector.

The following theorem is the analogue, for each ramification stratum and
with the additional rigidified framed decorations, of the
DeMarco--McMullen compactification of polynomial moduli by projectivized
polynomial trees; see \cite[Theorem 1.5]{DM08}. 

\begin{thm}\label{thm:main-compactification}
For every polynomial ramification profile $\vec\mu^*$, the
projectivized space of framed decorated polynomial trees
$\mathbb P \mathcal{T}_D(\vec\mu^*)$ gives a compactification of
${\rm Poly}_D(\vec\mu^*)$. More precisely, there is a compactification
$$
\overline{{\rm Poly}_D(\vec\mu^*)}^{\,T}
=
{\rm Poly}_D(\vec\mu^*)\sqcup \mathbb{P} \mathcal{T}_D(\vec\mu^*)
$$
whose boundary is homeomorphic to
$\mathbb P \mathcal{T}_D(\vec\mu^*)$. 
\end{thm}

Finally, we compare our compactification with the compactification of polynomial moduli by DeMarco--McMullen \cite{DM08}. Let $\mathcal{T}_D^{\mathrm{DM}}$ denote the DeMarco--McMullen space of polynomial trees of degree $D$, and let $$ \mathbb{P}\mathcal{T}_D^{\mathrm{DM}}:=(\mathcal{T}_D^{\mathrm{DM}}\setminus\{0\})/\mathbb R_{>0} $$ be its projectivization. For a polynomial ramification profile $\vec\mu^*=((D),\mu^2,\ldots,\mu^r)$, we denote by $ \mathcal{T}_D^{\mathrm{DM}}(\vec\mu^*)\subset \mathcal{T}_D^{\mathrm{DM}}$ the locus of DeMarco--McMullen polynomial trees whose non-simple critical values have ramification profiles $\lambda = \lambda(\vec{\mu}^*)$, counted on the tree with their natural local degrees. Here the multiset $\lambda(\vec{\mu}^*):=\{\mu^i_j-1: i \ge 2, \mu^i_j>1\}$ and $\mu^i=(\mu^i_1,\ldots,\mu^i_{\ell_i})$.
We similarly define $$ \mathbb{P}\mathcal{T}_D^{\mathrm{DM}}(\vec\mu^*) := \bigl(\mathcal{T}_D^{\mathrm{DM}}(\vec\mu^*)\setminus\{0\}\bigr)/\mathbb R_{>0}. $$ This locus is naturally a stratum of the DeMarco--McMullen tree space, but it is not closed in general. Indeed, when several finite critical points collide in a degenerating family, the ramification profile becomes coarser. Thus a face of $ \mathbb{P}\mathcal{T}_D^{\mathrm{DM}}(\vec\mu^*)$ may lie in $ \mathbb{P}\mathcal{T}_D^{\mathrm{DM}}(\vec\mu'^*) $ where $\vec\mu'^*$ is obtained from $\vec\mu^*$ by replacing some subcollection of parts by their sum. 

Let
$\mathbb P\mathcal T_D(\vec\mu^*)^{\mathrm{dyn}}
\subset
\mathbb P\mathcal T_D(\vec\mu^*)$
be the locus of projectivized framed decorated polynomial trees where after forgetting the rigidified framed Hurwitz decoration and keeping only the subtree spanned by the Julia set of the induced map on the boundary, the resulting infinite subtree has non-zero metric. 

Let $\mathbb P\mathcal T_D^{\mathrm{DM,adm}}(\vec\mu^*) \subset \mathbb P\mathcal T_D^{\mathrm{DM}}(\vec\mu^*)$ be the locus of projectivized DeMarco--McMullen polynomial trees whose local degree data are induced by an admissible cover degeneration with ramification profile $\vec\mu^*$. Equivalently, this is the locus where the local tropical Riemann--Hurwitz {\it equality} holds. This condition is automatic for the framed decorated trees constructed in this paper, but it is stronger than the local degree {\it inequality} in the definition of a DeMarco--McMullen polynomial-like tree.

\begin{thm} \label{thm:main-relation-to-DM}
For every polynomial ramification profile $\vec\mu^*$, there is a natural continuous map
$$
\Psi_{\vec\mu^*}\colon
\mathbb P\mathcal T_D(\vec\mu^*)^{\mathrm{dyn}}
\to
\mathbb P\mathcal T_D^{\mathrm{DM}}
$$
obtained by forgetting the rigidified framed Hurwitz decoration and
retaining only the underlying self-map on infinite trees with metric, then collapsing into the span of the Julia set at the boundary to obtain the projectivized DeMarco--McMullen polynomial tree.
Its image is contained in the union of DeMarco--McMullen strata
$$
\bigcup_{\vec\mu'^*\succeq \vec\mu^*}
\mathbb P\mathcal T_D^{\mathrm{DM}}(\vec\mu'^*),
$$
where $\vec\mu'^*\succeq \vec\mu^*$ if
$\vec\mu'^*$ is obtained from $\vec\mu^*$ by colliding some finite
critical points.

Moreover, the restriction of $\Psi_{\vec\mu^*}$ to the open
non-collision locus has image
$\mathbb P\mathcal T_D^{\mathrm{DM,adm}}(\vec\mu^*)$. 
The fibers of $\Psi_{\vec\mu^*}$ are, up to the Hurwitz
equivalence, given by the choices of rigidified framed Hurwitz decoration,
including the Glynn decoration and the rigidifying data
$(v,\ell_0,\epsilon)$, as well as the parts of the infinite metric tree that get collapsed to obtain the same underlying
DeMarco--McMullen polynomial tree.
\end{thm}

Therefore the compactification
$$
{\rm Poly}_D(\vec\mu^*)^T
=
{\rm Poly}_D(\vec\mu^*)\sqcup
\mathbb P\mathcal T_D(\vec\mu^*)
$$
is a decorated parameter space refinement of the DeMarco--McMullen
compactification of polynomial moduli. On the dynamical boundary locus,
forgetting decorations and passing from parameter spaces to moduli  spaces gives the
corresponding projectivized DeMarco--McMullen tree boundary, while purely
rigidifying boundary directions have no image in
$\mathbb P\mathcal T_D^{\mathrm{DM}}$.

\smallskip

A closely analogous phenomenon appears in the geometry of strata of
Abelian and quadratic differentials with fixed
orders of zeros and poles. Such a stratum carries more
structure than the underlying moduli space of curves, as it remembers the
differential, together with the zero and pole data, while the
Deligne--Mumford compactification remembers only the limiting stable
curve. Compactifications of strata of Abelian differentials, and more
generally of $k$-differentials, therefore refine the
Deligne--Mumford compactification by adding differential data on stable
curves, together with residue, level, and scaling information. This is
parallel to the relation in Theorem \ref{thm:main-relation-to-DM}. In both settings, boundary strata of a fixed combinatorial type may
meet strata of coarser type: for differentials this occurs when zeros or
poles collide or redistribute on components of a stable curve, while in
our setting it occurs when finite critical points collide and the
ramification profile becomes coarser. See, for example, the incidence
variety and multi-scale compactifications of strata of differentials in
\cite{BCGGMAbelian,BCGGMMultiscale,BCGGMKDiff,CMZKDiff,GendronIVC}.

\subsection{Strategy of proofs}
We begin with the study of three moduli spaces of covers: the Hurwitz space, the framed Hurwitz space, and the rigidified framed Hurwitz space. The Hurwitz space, its compactication and tropicalization have been well-studied; see Section \ref{sec:Hur}.
In particular, Cavalieri--Markwig--Ranganathan \cite{CMR16} and Glynn \cite{Glynn25} proved the {\it skeleton theorem} that the Berkovich skeleton is naturally identified with the corresponding tropical Hurwitz space; see Theorem \ref{thm_CMR16}.

For the framed Hurwitz space, its admissible cover compactification has been studied by Cavalieri \cite{Cavalieri06}. We introduce a tropical framed Hurwitz space and prove the corresponding skeleton theorem; see Theorem \ref{thm_framed_CMR16}.

The rigidified framed Hurwitz space is obtained from a framed Hurwitz space by adding rigidifying data on the source curve $\mathbb P^1$, in such a way that the source automorphism group is trivial. Then we show that the rigidified framed Hurwitz space admits a proper toroidal compactification. Similar to the framed Hurwitz space case, we introduce a rigidified framed tropical Hurwitz space and prove the corresponding skeleton theorem; see Theorem \ref{thm_rig_framed_CMR16}.  

To prove Theorem \ref{thm:main-1},
we show that the space $\mathcal{T}_D(\vec\mu^*)$ of trees is homeomorphic to the tropical rigidified framed Hurwitz space; see Theorem \ref{thm:comp-2}. This combined with Theorem \ref{thm_rig_framed_CMR16} gives Theorem \ref{thm:main-1}.

For Theorem \ref{thm:main-compactification}, we adapt the structure of the DeMarco--McMullen compactification of polynomial moduli by projectivized polynomial trees \cite[Theorem 1.5]{DM08}. We use the projectivized logarithmic rates at which polynomials approach the toroidal boundary of the rigidified framed Hurwitz space. These rates define projectivized tropical limits in $\mathbb P\mathcal T_D(\vec\mu^*)$, independently of the choice of toroidal coordinates. We show that every divergent sequence in ${\rm Poly}_D(\vec\mu^*)$ has such a limit after passing to a subsequence, and conversely that every point of $\mathbb P\mathcal T_D(\vec\mu^*)$ is realized by smoothing an appropriate toroidal boundary stratum. This allows us to attach $\mathbb P\mathcal T_D(\vec\mu^*)$ as the compactifying boundary of ${\rm Poly}_D(\vec\mu^*)$.

For Theorem \ref{thm:main-relation-to-DM}, we compare this decorated compactification with the DeMarco--McMullen compactification of polynomial moduli. One essential distinction is that $\mathbb P\mathcal T_D(\vec\mu^*)$ compactifies a parameter stratum and therefore remembers rigidified framed Hurwitz data, while the DeMarco--McMullen space remembers only the underlying projectivized polynomial tree in moduli. Furthermore, the infinite trees for elements of $\mathbb P\mathcal T_D(\vec\mu^*)$ are invariant subtrees that contains all the critical points hence might be larger than the DeMarco--McMullen polynomial tree. We show that on the dynamical locus
$\mathbb P\mathcal T_D(\vec\mu^*)^{\rm dyn}
\subset
\mathbb P\mathcal T_D(\vec\mu^*)$, forgetting the Glynn decoration, the framed target component, the marked source point, and the tangent vector rigidification, then do a collapsing to a subtree spanned by the Julia set at the boundary (type-I Julia set), gives the desired map $\Psi_{\vec\mu^*}$. Its image lies in the union of strata corresponding to profiles obtained from $\vec\mu^*$ by collisions of finite critical points, while on the open non-collision locus the image is exactly the DeMarco--McMullen stratum of profile $\vec\mu^*$ whose local degree data are induced by an admissible cover degeneration.

\subsection*{Organization of the paper} The paper is organized as follows.
In Section \ref{sec:prelim}, we review Hurwitz spaces and prove the framed analogue of the skeleton theorem for framed Hurwitz spaces. In Section 3, we introduce the rigidified framed Hurwitz space, construct a proper toroidal compactification, and prove the rigidified framed analogue of the skeleton theorem. In Section \ref{sec:4-trees}, we define framed decorated polynomial trees and prove Theorem \ref{thm:main-1}. In Section \ref{sec:pf-1.2}, we prove Theorem \ref{thm:main-compactification}, and in Section \ref{sec:pf-1.3}, we prove Theorem \ref{thm:main-relation-to-DM}.

\section{Preliminaries}\label{sec:prelim}
In this section, we first review Hurwitz spaces, their admissible cover compactifications and the associated Berkovich skeleta, and the tropical Hurwitz spaces of Cavalieri--Markwig--Ranganathan \cite{CMR16} and Glynn \cite{Glynn25} in Section \ref{sec:Hur}. Then in Section \ref{sec_fhur}, we discuss the framed Hurwitz space, which is the intermediate moduli space from which the rigidified framed Hurwitz space of Section 3 will be obtained. We define the framed tropical Hurwitz space and prove the framed skeleton theorem; see Theorem \ref{thm_framed_CMR16}.

\subsection{Hurwitz spaces, Berkovich skeleton and tropical Hurwitz spaces}\label{sec:Hur}
Fix a vector $\vec{\mu} = (\mu^1,\ldots,\mu^r)$ of partitions of $D$. We regard the partitions $\mu^i$ as the prescribed non-simple ramification profiles.
Denote by $\mathcal{H}_{g,h,D}(\vec{\mu})$ the Hurwitz space of degree $D$ Hurwitz covers $[C_2 \to C_1]$ of smooth genus $h$ curves by genus $g$ curves with ramification $\mu^i$ over smooth marked points $p_i$ of $C_1$ and simple ramification over smooth marked points $q_1,\ldots,q_s$. Here $\mu^i$ is the vector of multiplicities of preimages of $p_i$ and simple ramification means that the ramification profile over each $q_j$ is $ (2,1^{D-2})$. Equivalently, Hurwitz space parametrizes holomorphic maps from Riemann surface $C_2$ to Riemann surface $C_1$ with prescribed ramification profile $\vec{\mu}$, up to pre- and post-composition by isomorphisms.

In \cite{HM82}, Harris--Mumford gave a compactification of the Hurwitz space $\mathcal{H}_{g, h, D}(\Vec{\mu})$ by the space of {\it admissible covers}. The space of admissible covers is in general not normal; however, its normalization is always smooth and admits an interpretation as the stack of twisted stable maps by Abramovich--Corti--Vistoli \cite{ACV03}. We refer to the normalized stack as the stack of admissible covers $\overline{\mathcal{H}}_{g,h,D}(\Vec{\mu})$, which is a smooth compactification of the Hurwitz space $\mathcal{H}_{g,h,D}(\Vec{\mu})$. 

\begin{rmk} \label{rmk_toroidal}
By \cite[Section 3.1.1]{CMR16}, the stack $\overline{\mathcal{H}}_{g,h,D}(\Vec{\mu})$ is a smooth Deligne--Mumford stack and the boundary $\overline{\mathcal{H}}_{g,h,D}(\Vec{\mu}) \setminus \mathcal{H}_{g,h,D}(\Vec{\mu})$ is a divisor with normal crossings. Therefore, the inclusion $\mathcal{H}_{g,h,D}(\Vec{\mu}) \to \overline{\mathcal{H}}_{g,h,D}(\Vec{\mu})$ is a {\it toroidal embedding} of Deligne--Mumford stacks.
\end{rmk}

\medskip

Given a toroidal embedding $U \to X$ of Deligne--Mumford stacks, there is an associated cone complex with integral structure called the {\it skeleton} $\Sigma(X)$.
Abramovich--Caporaso--Payne \cite{ACP15}, building on the work of Thuillier \cite{Thuillier07}, shows that this cone complex naturally lives in the Berkovich analytification $X^{\rm an}$ of $X$ by showing that there is a {\it retraction} map
$\mathcal{R}_X \colon X^{an} \to X^{an}$ whose image is $\overline{\Sigma}(X)$.

Let $\overline{\mathcal{H}}_{g,h,D}^{\rm an}(\Vec{\mu})$ be the Berkovich analytification of $\overline{\mathcal{H}}_{g,h,D}(\Vec{\mu})$, where the field $\mathbb{C}$ is equipped with the trivial (non-Archimedean) norm. By Remark 
\ref{rmk_toroidal}, 
let $\overline{\Sigma}(\overline{\mathcal{H}}^{\rm an}_{g,h,D}(\Vec{\mu}))$ be the skeleton of $\overline{\mathcal{H}}_{g,h,D}(\Vec{\mu})$ as constructed by Thuillier \cite{Thuillier07}.

\medskip

Cavalieri--Markwig--Ranganathan constructed a tropical moduli space $\mathcal{H}_{g,h,D}^{\rm trop}(\Vec{\mu})$ of admissible covers and a natural surjective face morphism from the Berkovich skeleton of the admissible cover compactification to this tropical space; see \cite[Theorem 1]{CMR16}. Glynn \cite{Glynn25} later refined the combinatorial structure, for Hurwitz spaces of maps to $\mathbb{P}^1$, by using decorated trees that index the irreducible boundary strata; with this refinement, the associated tropical Hurwitz space is isomorphic to the skeleton; see \cite[Corollary 1.4]{Glynn25}.
Now we summarize Glynn's construction of a tropical Hurwitz space $\overline{\mathcal{H}}^{\rm trop}(\Vec{\mu})$.

\subsubsection{Glynn's Tropical Hurwitz spaces}
Let $B$ be a finite subset of $S^2$ with $|B| \ge 3$, corresponding to the set of critical points. A {\it $B$-marked stable tree} is a connected tree $T$ with $|B|$ legs together with a labelling from $B$ to the legs; see \cite[Section 2.2]{Glynn25}.
Let $T$ be a $B$-marked stable tree. Glynn defined a $\Vec{\mu}$-decoration of $T$; see \cite[Definition 4.1]{Glynn25}. Roughly speaking, a decoration of $T$ records certain {\it monodromy} information which comes from embedding $T$ into $S^2$ of a branched cover of $S^2$. A $B$-marked stable tree $T$ with a $\Vec{\mu}$-decoration is called a {\it $\Vec{\mu}$-decorated tree}.

Two $\Vec{\mu}$-decorations of a stable tree $T$ are {\it Hurwitz equivalent} if one can be obtained from the other via a sequence of global conjugations and braid moves of decorated trees; see \cite[Section 4.3]{Glynn25} for definitions. 
We denote by
$${\rm Stab}_{0,B}(\Vec{\mu})$$
the set of Hurwitz equivalence classes of $\Vec{\mu}$-decorated trees.

Glynn defined the {\it tropical Hurwitz space} $\overline{\mathcal{H}}^{\rm trop}(\Vec{\mu})$ as follows; see \cite[Section 8]{Glynn25} for details.
For each $\Theta \in {\rm Stab}_{0,B}(\Vec{\mu})$, we define the extended cone
$$\overline{\sigma}_{\Theta} := (\mathbb{R}_{\ge 0} \cup \{\infty\})^{|E(T)|}.$$
Define
$$\overline{\mathcal{H}}^{\rm trop}(\Vec{\mu}) := \lim_{\rightarrow}(\overline{\sigma}_{\Theta},j_{\omega})$$
where $\Theta$ ranges over elements in ${\rm Stab}_{0,B}(\Vec{\mu})$ and $j_\omega$ ranges over edge contractions of $T$.

Combining results of Cavalieri--Markwig--Ranganathan and Glynn gives the following theorem.

\begin{thm}[{\cite[Theorem 1]{CMR16},\cite[Corollary 1.4]{Glynn25}}] \label{thm_CMR16}
We have the following commutative diagram
$$
\xymatrix{
 \overline{\mathcal{H}}_{g, 0, D}^{\rm an}(\Vec{\mu}) \ar[d]_{\mathrm{trop}} \ar[r]^{\mathcal{R}\quad} & \overline{\Sigma}( \overline{\mathcal{H}}_{g, 0, D}^{\rm an}(\Vec{\mu})) \ar[dl]^{{\rm trop}_\Sigma}  \\ 
 \overline{\mathcal{H}}^{\rm trop}(\Vec{\mu}) &
}
$$
where $\mathcal R$ is the Berkovich retraction onto the skeleton. 
Moreover, the map ${\rm trop}_\Sigma \colon \overline{\Sigma}(\overline{\mathcal{H}}^{\rm an}_{g, 0, D}(\Vec{\mu})) \to \overline{\mathcal{H}}^{\rm trop}(\Vec{\mu})$ is an isomorphism of extended cone complexes.
\end{thm}

\subsection{Framed Hurwitz spaces}\label{sec_fhur}
We denote by $\mathcal{H}_{g,\mathbb{P}^1,D}(\vec{\mu})$ the {\it framed Hurwitz space} of degree $D$ Hurwitz covers $[C_2 \to \mathbb{P}^1]$ of a parametrized $\mathbb{P}^1$ by genus $g$ curves $C_2$ with ramification profile $\vec{\mu}$. In other words, a framed Hurwitz space is 
a Hurwitz space where we do not quotient by $\mathrm{Aut}(\mathbb{P}^1)$ on the base.

By the work of Cavalieri \cite{Cavalieri06}, the Abramovich--Corti--Vistoli \cite{ACV03} theory still gives a twisted admissible cover compactification $\overline{\mathcal{H}}_{g,\mathbb{P}^1,D}(\vec{\mu})$, which is the (smooth) stack denoted by $$\overline{{\rm Adm}}\!\left(g \xrightarrow{D} \mathbb{P}^1, (\mu^1,\dots,\mu^r)\right).$$

\begin{rmk}
In Cavalieri's compactification, the target is allowed to degenerate to a semistable expansion of the parametrized $\mathbb{P}^1$. Equivalently, boundary points of $\overline{\mathcal H}_{g,\mathbb{P}^1,D}(\vec{\mu})$ correspond to admissible covers of expanded genus-zero targets together with a distinguished parametrized component. Thus, although the open framed Hurwitz space parametrizes covers of a smooth parametrized $\mathbb{P}^1$, its compactification includes degenerations in which the target acquires additional rational components.
\end{rmk}

\begin{lem}\label{lem_frame_proper}
The compactified framed Hurwitz stack $\overline{\mathcal{H}}_{g,\mathbb P^1,D}(\vec\mu)$ is a proper Deligne--Mumford stack.
\end{lem}

\begin{proof}
By definition, Cavalieri's compactification
$$
\overline{\mathcal{H}}_{g,\mathbb P^1,D}(\vec\mu)
=
\overline{{\rm Adm}}\!\left(g \xrightarrow{D} \mathbb P^1,(\mu^1,\dots,\mu^r)\right)
$$
is the Abramovich--Corti--Vistoli stack of balanced twisted admissible covers with parametrized genus-zero target and ordered branch points. By the properness theorem for twisted stable maps of Abramovich--Vistoli \cite{AV02} (see also \cite[Theorem 2.1.7]{ACV03}), it is a proper Deligne--Mumford stack.
\end{proof}

\begin{lem}\label{lem_frame_toroidal_bdry}
The compactified framed Hurwitz stack
$\overline{\mathcal{H}}_{g,\mathbb P^1,D}(\vec{\mu})$
is a toroidal Deligne--Mumford stack.
\end{lem}

\begin{proof}
Let $x \in \overline{\mathcal{H}}_{g,\mathbb P^1,D}(\vec{\mu})$ be a geometric point corresponding to a framed twisted admissible cover
$[\phi \colon \mathcal C \to \mathcal P]$
with semistable target $\mathcal P$, where the target $\mathbb P^1$ is parametrized and the branch
points are ordered.

We describe the completed local ring of $\overline{\mathcal{H}}_{g,\mathbb P^1,D}(\vec{\mu})$ at $x$.
Let $z_1,\dots,z_k$ be the nodes of the target $\mathcal P$. For each $i$, let $\widetilde z_{i,1},\dots,\widetilde z_{i,r(i)}$ be the nodes of $\mathcal C$ lying above $z_i$, with corresponding ramification indices
$p(i,1),\dots,p(i,r(i))$.
Let $\xi_i$ denote the smoothing parameter of the target node $z_i$, and let $\xi_{i,j}$ denote the smoothing parameter of the source node $\widetilde z_{i,j}$.
Let $\eta_1,\dots,\eta_m$ be local parameters along the stratum.

By the standard local deformation theory of admissible covers (see \cite[Section 4.2.2]{CMR16}), the completed local ring at $x$ is analytically isomorphic to
$$
\widehat{\mathcal O}_{\overline{\mathcal{H}}_{g,\mathbb P^1,D}(\vec{\mu}),x}
\cong
\mathbb C[[\xi_1,\dots,\xi_k,\eta_1,\dots,\eta_m,\xi_{i,j}]]
\Big/
\left(\xi_{i,j}^{\,p(i,j)}-\xi_i\right)_{i,j}.
$$
Here the variables $\xi_i$ record smoothing of the nodes of the target, the variables $\xi_{i,j}$
record smoothing of nodes of the source, and the $\eta_\ell$ are parameters tangent to the stratum. For each $i$, the source node parameters $\xi_{i,j}$ are constrained by the relations $\xi_{i,j}^{\,p(i,j)}=\xi_i,$
so they are integral over the submonoid generated by the target smoothing parameters $\xi_i$. In particular, they do not contribute additional independent boundary directions beyond those already coming from the nodes of the target.

These equations are monomial. In particular, \'etale locally the completed local ring is the completed
local ring of an affine toric variety times a smooth factor. Moreover, the boundary divisor is cut out
by the vanishing of the node-smoothing parameters $\xi_i$, equivalently by the union of the
coordinate hypersurfaces corresponding to the toric boundary. Thus the pair
$$
\left(\overline{\mathcal{H}}_{g,\mathbb P^1,D}(\vec{\mu}),
\overline{\mathcal{H}}_{g,\mathbb P^1,D}(\vec{\mu}) \setminus \mathcal{H}_{g,\mathbb P^1,D}(\vec{\mu})\right)
$$
is toroidal near $x$.

Since $x$ is arbitrary, $\overline{\mathcal{H}}_{g,\mathbb P^1,D}(\vec{\mu})$ is a toroidal Deligne--Mumford stack. This completes the proof.
\end{proof}

\begin{rmk}
By Lemmas \ref{lem_frame_proper} and \ref{lem_frame_toroidal_bdry}, the results of Abramovich--Caporaso--Payne \cite{ACP15}
apply to the framed compactification
$
\overline{\mathcal{H}}_{g,\mathbb P^1,D}(\vec{\mu}).
$
Hence its Berkovich analytification
$\overline{\mathcal{H}}^{\rm an}_{g,\mathbb P^1,D}(\vec{\mu})$
admits a canonical skeleton
$\overline{\Sigma}\!\left(\overline{\mathcal{H}}^{\rm an}_{g,\mathbb P^1,D}(\vec{\mu})\right)$,
which is an extended cone complex canonically determined by the toroidal boundary stratification.
\end{rmk}

\subsubsection{Framed tropical Hurwitz spaces}
We now define the tropical moduli space corresponding to framed Hurwitz covers, i.e., Hurwitz covers
with parametrized target $\mathbb P^1$.

\begin{defn}\label{def:trop-framed}
A \emph{framed $B$-marked tree} is a connected tree $T$ with legs labelled by $B$, together with a distinguished vertex $v\in V(T)$, such that every vertex other than $v$ is stable, i.e., has valence at least three. The distinguished vertex $v$ represents the parametrized component of the expanded target $\mathbb P^1$ and
is allowed to be unstable.

A \emph{framed $\vec{\mu}$-decorated tree} is a triple $(T,\delta,v)$, where $(T,v)$ is a framed $B$-marked stable tree, $\delta$ is a $\vec{\mu}$-decoration of $T$ in the sense of Glynn. 

Two framed $\vec{\mu}$-decorated trees $(T,\delta,v)$ and $(T,\delta',v)$ are said to be \emph{framed Hurwitz equivalent} if $\delta$ and $\delta'$ are Hurwitz equivalent in the sense of Glynn \cite[Section 4]{Glynn25}. Let 
$${\rm Stab}^{\mathrm{fr}}_{0,B}(\vec{\mu})$$ denote the set of framed Hurwitz equivalence classes of framed $\mu$-decorated trees.

For $\Theta=[(T,\delta,v)] \in {\rm Stab}^{\mathrm{fr}}_{0,B}(\vec{\mu})$, define
$$
\overline{\sigma}_\Theta := (\mathbb R_{\geq 0}\cup\{\infty\})^{E(T)}.
$$
If $\Theta'=[(T',\delta',v')]$ is obtained from $\Theta=[(T,\delta,v)]$ by contracting a collection of edges of $T$, with distinguished vertex $v'$ equal to the image of $v$ under the contraction, then there is a natural face morphism
$$
j_{\Theta',\Theta}\colon \overline{\sigma}_{\Theta'} \hookrightarrow \overline{\sigma}_\Theta.
$$
The \emph{framed tropical Hurwitz space} is the extended cone complex
$$
\overline{\mathcal{H}}^{\rm trop}_{\mathbb P^1}(\vec{\mu})
:=
\varinjlim_{\Theta \in {\rm Stab}^{\mathrm{fr}}_{0,B}(\mu)}
(\overline{\sigma}_\Theta, j_{\Theta',\Theta}).
$$
\end{defn}

We now prove the framed analogue of Theorem \ref{thm_CMR16}.

\begin{thm}\label{thm_framed_CMR16}
There is a commutative diagram
$$
\begin{tikzcd}
\overline{\mathcal{H}}^{\rm an}_{g,\mathbb P^1,D}(\vec{\mu})
  \arrow[r, "\mathcal R"]
  \arrow[d, "\rm trop_{\mathbb P^1}"']
&
\overline{\Sigma}\!\left(\overline{\mathcal{H}}^{\rm an}_{g,\mathbb P^1,D}(\vec{\mu})\right)
  \arrow[dl, "\rm trop_{\Sigma,\mathbb P^1}"]
\\
\overline{\mathcal{H}}_{\mathbb P^1}^{\rm trop}(\vec{\mu})
&
\end{tikzcd}
$$
where $\mathcal R$ is the Berkovich retraction onto the skeleton. Moreover,
${\rm trop}_{\Sigma,\mathbb P^1} \colon
\overline{\Sigma}\!\left(\overline{\mathcal H}^{\rm an}_{g,\mathbb P^1,D}(\vec{\mu})\right)
\to
\overline{\mathcal{H}}_{\mathbb P^1}^{\rm trop}(\vec{\mu})$
is an isomorphism of extended cone complexes.
\end{thm}

\begin{proof}
By Lemmas \ref{lem_frame_proper} and \ref{lem_frame_toroidal_bdry}, the stack
$\overline{\mathcal{H}}_{g,\mathbb P^1,D}(\vec{\mu})$
is proper and toroidal. Therefore, by \cite{ACP15}, its analytification admits a canonical skeleton.

We first define the tropicalization map. Let
$x \in \overline{\mathcal{H}}^{\rm an}_{g,\mathbb P^1,D}(\vec{\mu})$
be an analytic point represented by a family of framed admissible covers over a valuation ring.
Its special fiber determines a boundary stratum of the compactification, hence a decorated dual graph
of the degeneration. By Definition \ref{def:trop-framed}, this framed decorated dual graph
determines a point in the cone of
$\overline{\mathcal{H}}^{\rm trop}_{\mathbb{P}^1}(\vec{\mu})$ indexed by the framed class $[T,\delta,v]$.
Taking the valuations of the node-smoothing parameters defines a continuous tropicalization map
${\rm trop}_{\mathbb P^1} \colon
\overline{\mathcal{H}}^{\rm an}_{g,\mathbb P^1,D}(\vec{\mu}) \to \overline{\mathcal{H}}_{\mathbb P^1}^{\rm trop}(\vec{\mu})$.

Next we show that ${\rm trop}_{\mathbb P^1}$ factors through the skeleton. By Lemma \ref{lem_frame_toroidal_bdry}, \'etale locally
near a boundary point the completed local ring is given by monomial admissible cover equations
$$
\xi_{i,j}^{\,p(i,j)}=\xi_i.
$$
Hence the characteristic monoid of the logarithmic structure is generated by the smoothing parameters
of the nodes of the target, and the associated cone of the skeleton is exactly the cone of valuations
of these smoothing parameters. It follows that $\rm trop_{\mathbb P^1}$ depends only on the corresponding
point of the skeleton, so it factors through the retraction
$$
\mathcal R \colon
\overline{\mathcal{H}}^{\rm an}_{g,\mathbb P^1,D}(\vec{\mu})
\to
\overline{\Sigma}\!\left(\overline{\mathcal{H}}^{\rm an}_{g,\mathbb P^1,D}(\vec{\mu})\right).
$$
This yields a map
$$
{\rm trop}_{\Sigma,\mathbb P^1} \colon
\overline{\Sigma}\!\left(\overline{\mathcal{H}}^{\rm an}_{g,\mathbb P^1,D}(\vec{\mu})\right)
\to
\overline{\mathcal{H}}_{\mathbb P^1}^{\rm trop}(\vec{\mu}).
$$

Finally, we identify the two extended cone complexes. A toroidal boundary stratum of
$\overline{\mathcal{H}}_{g,\mathbb P^1,D}(\vec{\mu})$ corresponds to a framed admissible cover degeneration. Its target
dual tree is a $B$-marked stable tree $T$, the admissible cover determines a $\vec{\mu}$-decoration
$\delta$ of $T$ in the sense of Glynn, and the distinguished parametrized component of the target
determines a distinguished internal vertex $v$ of $T$. Thus boundary strata are indexed by the
same framed combinatorial types $(T,\delta,v)$ that index the cones of
$H^{\rm trop}_{\mathbb P^1}(\vec{\mu})$. Under specialization, nodes are smoothed or created exactly by
contracting edges of $T$, and the distinguished vertex is sent to its image under contraction.
Hence the face morphisms agree on both sides. Therefore
$\overline{\Sigma}(\overline{\mathcal{H}}^{\rm an}_{g,\mathbb P^1,D}(\vec{\mu}))\cong \overline{\mathcal{H}}^{\rm trop}_{\mathbb P^1}(\vec{\mu})$.
\end{proof}

\section{Rigidified framed Hurwitz spaces}\label{sec_rfhur}
In this section, we introduce the rigidified framed Hurwitz space in Section 3.1. We show that it admits a proper toroidal compactification in Section 3.2. Finally, in Section 3.3, we define the rigidified framed tropical Hurwitz space and prove the rigidified framed version of Theorem \ref{thm_CMR16}.

Throughout this section, we take the source curve to be genus-$0$ and set $\Vec{\mu}^*:=((D),\mu^2,\ldots,\mu^r)$.

\subsection{Rigidified framed Hurwitz spaces and spaces of polynomials}
We further {\it rigidify} the framed Hurwitz space $\mathcal{H}_{0,\mathbb{P}^1,D}(\Vec{\mu}^*)$ by fixing in the source curve the point $x_\infty$, a tangent vector $v_\infty$ at $x_\infty$ and a point $x_0 \neq x_\infty$.
Then an automorphism of the source $\mathbb{P}^1$ fixing $x_\infty$, $v_\infty$ and a point $x_0 \neq x_\infty$
must be the identity. Indeed, choose an affine coordinate $z$ on $\mathbb{P}^1$ such that $x_\infty=\infty$. Then any automorphism of $\mathbb{P}^1$ fixing $\infty$ is of the form $z \mapsto az+b$.
Fixing the tangent vector at $\infty$ forces $a=1$, and fixing $x_0$ forces $b=0$.

We denote by the rigidified framed Hurwitz space by $\mathcal{H}_{{\rm rig},0,\mathbb{P}^1,D}(\Vec{\mu}^*)$. 

\begin{rmk} \label{rmk_polyD}
By the above, $\mathcal{H}_{{\rm rig},0,\mathbb{P}^1,D}(\Vec{\mu}^*)$ is the space of degree $D$ polynomials with ramification profile $\Vec{\mu}^*$, after choosing the affine coordinates determined by the rigidifying data.
\end{rmk}

\subsection{Twisted admissible cover compactification} \label{sec_3.2_rfhur}
We denote by $\overline{\mathcal{H}}_{{\rm rig},0,\mathbb{P}^1,D}(\vec{\mu}^*)$ the twisted admissible cover compactification of the rigidified framed Hurwitz space $\mathcal{H}_{{\rm rig},0,\mathbb{P}^1,D}(\Vec{\mu}^*)$.

To construct the space $\overline{\mathcal{H}}_{{\rm rig},0,\mathbb{P}^1,D}(\vec{\mu}^*)$, 
we use the short-hand notation $$X:=\overline{\mathcal{H}}_{0,\mathbb{P}^1,D}(\vec{\mu}^*).$$
For a framed admissible cover in $X$, we denote by the {\it distinguished source component} the irreducible component of the source curve containing the
marked point $x_\infty$ lying over the totally ramified branch point at infinity.
Define $X_{x_0}$ to be the stable compactification obtained by allowing the extra marked point
$x_0$ to specialize to the special points and then stabilizing the source curve.

\begin{lem}\label{lem_Xx0_prop}
The stack $X_{x_0}$ has toroidal boundary and is proper over $X$.
\end{lem}

\begin{proof}
We first prove properness. Let $R$ be a valuation ring with fraction field $K$, and suppose we are
given a commutative diagram
$$
\begin{tikzcd}
{\rm Spec} K \arrow[r] \arrow[d] & X_{x_0} \arrow[d] \\
{\rm Spec} R \arrow[r] & X.
\end{tikzcd}
$$
Since $X$ is proper by Lemma \ref{lem_frame_proper}, after a finite base change the corresponding family of framed admissible covers
over ${\rm Spec} K$ extends to a family over ${\rm Spec} R$. Over the generic fiber, we also have the extra
marked point $x_0$ on the distinguished source component. If its limit remains in the smooth locus
away from the special points, then it extends directly. If its limit meets a node or an existing marked
point, one inserts a rational bubble carrying the marked point $x_0$, and then stabilizes. This is the usual stable marked curve compactification procedure, and it produces a unique stable
limit. Hence $X_{x_0}\to X$ satisfies the valuative criterion for properness.

We next prove toroidality. Let $y\in X_{x_0}$ be a boundary point, and let $x\in X$ be its image.
Choose an \'etale neighborhood $U\to X$ of $x$ with toroidal coordinates $\xi_1,\dots,\xi_r,u_1,\dots,u_m$
such that the boundary of $X$ is given by
$$
\xi_1\cdots \xi_r=0.
$$
Passing from $X$ to $X_{x_0}$ amounts to adding one extra marked point on the distinguished
source component and taking the relative stable marked curve compactification. Locally, this behaves
as in the universal semistable marked curve construction over a toroidal base. Namely, if $x_0$ stays in the
smooth locus away from the special points, one simply adjoins an ordinary smooth coordinate $t$;
if $x_0$ specializes to the special locus, one obtains in addition one boundary parameter
corresponding to the bubbling or specialization of $x_0$.

Thus, \'etale locally, the completed local ring of $X_{x_0}$ is obtained from that of $X$ by adjoining
ordinary smooth variables and, in the boundary case, at most one additional monomial coordinate.
Therefore the boundary divisor of $X_{x_0}$ is again locally given by coordinate hyperplanes in a
toroidal chart. Hence $X_{x_0}$ is toroidal.
\end{proof}

Let $L$ be the tangent line bundle at $x_\infty$ on the universal distinguished component over $X_{x_0}$. Then $L^\times:= L \setminus \{0\}$ 
is the moduli stack of non-zero tangent vectors, as  choosing a non-zero tangent vector is equivalent to choosing a point of the complement of the zero section.

\begin{lem}
$L^\times$ has toroidal boundary but is in general not proper.
\end{lem}
\begin{proof}
Since line bundles are Zariski locally trivial, \'etale locally on $X_{x_0}$, we have $L^\times \cong X_{x_0}\times \mathbf G_m$ where $\mathbf G_m$ is a torus. Therefore $L^\times$ is toroidal. However, $L^\times$ is in general not
proper over $X_{x_0}$ as the fiber $\mathbf G_m$ is not proper. 
\end{proof}

\begin{lem}\label{lem_fiber_comp}
Any toric compactification of the $\mathbf G_m$-fiber yields a proper compactification whose boundary is toroidal. 
\end{lem}
\begin{proof}
The $\mathbb P^1$-bundle $\mathbb P(\mathcal O\oplus L)\to X_{x_0}$ is a toric compactification of $L^\times$.
\'Etale locally where $L$ is trivial, the $\mathbb P^1$-bundle is the product $X_{x_0}\times \mathbb P^1$ and we have $L^\times \cong X_{x_0}\times \mathbf G_m \subset X_{x_0}\times \mathbb P^1$.

The added boundary is the union of the two relative toric divisors $X_{x_0}\times \{0\}$ and $X_{x_0}\times \{\infty\}$.
Thus, \'etale locally, the total boundary is the union of the original toroidal boundary $\{\xi_1\cdots \xi_r=0\}$
together with the coordinate hyperplanes coming from the toric boundary of
$\mathbb P^1$. Hence $\mathbb P(\mathcal O\oplus L)$ has toroidal boundary. Finally, $\mathbb P(\mathcal O\oplus L)$ is proper as the projective bundle is a proper map and $X_{x_0}$ is proper over $X$ by Lemma \ref{lem_Xx0_prop}.
\end{proof}

\begin{rmk}
We denote by $\overline{\mathcal{H}}_{{\rm rig},0,\mathbb{P}^1,D}(\vec{\mu}^*)$ a chosen proper toroidal compactification as in Lemma \ref{lem_fiber_comp}. Thus the compactification is not canonically determined and it depends on the choice of toric compactification of the $\mathbf G_m$-fiber of $L^\times\to X_{x_0}$. We choose the $\mathbb{P}^1$-bundle compactification. Then again, Abramovich–Caporaso–Payne \cite{ACP15} gives the skeleton of $\overline{\mathcal{H}}_{{\rm rig},0,\mathbb{P}^1,D}(\vec{\mu}^*)$ which lives in the Berkovich analytification $\overline{\mathcal{H}}^{\rm an}_{{\rm rig},0,\mathbb{P}^1,D}(\vec{\mu}^*)$.
\end{rmk}

\subsection{Rigidified framed tropical Hurwitz spaces}

\begin{defn}\label{def:rigidified_tropical_space}
A \emph{rigidified framed $\vec\mu^*$-decorated type} is a tuple
$$
\Xi=(T,\delta,v,\Gamma_{\rm src}^+,\ell_0,\epsilon,\rho),
$$
where:
\begin{itemize}
\item $(T,\delta,v)$ is a framed $\vec\mu^*$-decorated tree in the sense of Definition \ref{def:trop-framed};

\item $\Gamma_{\rm src}\to T$ is the source side combinatorial admissible cover associated to the framed Hurwitz equivalence class of $(T,\delta,v)$.

    \item $\Gamma_{\rm src}^+$ is the stabilized source graph obtained from $\Gamma_{\rm src}$ after adding an additional marked leg $\ell_0$, representing the marked source point $x_0$.

\item $\epsilon\in\{0,1\}$ records whether the stabilization of $x_0$ contributes an additional boundary modulus:
\begin{itemize}
\item $\epsilon=0$ if $x_0$ lies in the smooth locus away from the special points,
\item $\epsilon=1$ if $x_0$ specializes to the special locus and contributes an additional boundary parameter of Lemma \ref{lem_Xx0_prop}.
\end{itemize}

\item $\rho$ is a cone in the fan of the chosen toric compactification of the $\mathbf G_m$-bundle of non-zero tangent vectors at $x_\infty$. For the compactification
$\mathbb P(\mathcal O\oplus L)\to X_{x_0}$,
this is the fan of $\mathbb P^1$. Thus $\rho$ records whether the tangent vector remains in the open torus or specializes to one of the two toric boundary divisors.
\end{itemize}

Two rigidified framed $\vec\mu^*$-decorated types
$$
\Xi=(T,\delta,v,\Gamma_{\rm src}^+,\ell_0,\epsilon,\rho)
\qquad \text{and}\qquad
\Xi'=(T',\delta',v',(\Gamma_{\rm src}')^+,\ell_0',\epsilon',\rho')
$$
are \emph{rigidified Hurwitz equivalent} if $(T,\delta,v)$ and $(T',\delta',v')$ are framed Hurwitz equivalent, the induced identification of the stabilized source side combinatorial admissible covers sends $\ell_0$ to $\ell_0'$, $\epsilon=\epsilon'$, and
$\rho=\rho'$.

Let
$
{\rm Stab}^{\rm rig,fr}_{0,B}(\vec\mu^*)
$
denote the set of rigidified Hurwitz equivalence classes of rigidified framed $\vec\mu^*$-decorated types.
For
$
\Xi\in {\rm Stab}^{\rm rig,fr}_{0,B}(\vec\mu^*),
$
define the associated extended cone by
$$
\overline{\sigma}^{\rm rig}_\Xi
:=
(\mathbb R_{\geq 0}\cup\{\infty\})^{E(T)}
\times
(\mathbb R_{\geq 0}\cup\{\infty\})^{\epsilon}
\times
\overline{\rho},
$$
where
the first factor
records the smoothing parameters of the nodes of the framed target, the second factor
records the additional boundary parameter introduced by the stabilization of the marked source point $x_0$, and $\overline{\rho} :=\mathbb R_{\ge 0}\cup\{\infty\}$
denotes the extended one-dimensional cone associated to the chosen toric
compactification of the $\mathbf G_m$-fiber of the nonzero tangent vector
bundle at $x_\infty$.

If $\Xi'$ is obtained from $\Xi$ by a specialization of rigidified framed decorated types, namely by contracting edges of $T$, sending the distinguished vertex $v$ to its image, stabilizing the source graph with the marked leg $\ell_0$, and specializing the tangent vector cone $\rho$, then there is a natural face morphism
$$
j^{\rm rig}_{\Xi',\Xi} \colon 
\overline{\sigma}^{\rm rig}_{\Xi'}
\hookrightarrow
\overline{\sigma}^{\rm rig}_{\Xi}.
$$

The \emph{rigidified framed tropical Hurwitz space} is the extended cone complex
$$
\overline{\mathcal H}^{\rm trop}_{\rm rig,0,\mathbb P^1,D}(\vec\mu^*)
:=
\varinjlim_{\Xi\in {\rm Stab}^{\rm rig,fr}_{0,B}(\vec\mu^*)}
\left(
\overline{\sigma}^{\rm rig}_\Xi,
j^{\rm rig}_{\Xi',\Xi}
\right).
$$
\end{defn}

We have a rigidified framed version of Theorem \ref{thm_CMR16}.
\begin{thm} \label{thm_rig_framed_CMR16}
There is a commutative diagram
$$
\begin{tikzcd}
\overline{\mathcal{H}}^{{\rm an}}_{\mathrm{rig},0,\mathbb P^1,D}(\vec\mu^*)
  \arrow[r,"\mathcal R"]
  \arrow[d,"{\rm trop}_{\mathrm{rig}}"']
&
\overline{\Sigma}\!\left(\overline{\mathcal{H}}^{{\rm an}}_{\mathrm{rig},0,\mathbb P^1,D}(\vec\mu^*)\right)
  \arrow[dl,"{\rm trop}_{\Sigma,\mathrm{rig}}"]
\\
\overline{\mathcal{H}}^{{\rm trop}}_{\mathrm{rig},0,\mathbb P^1,D}(\vec\mu^*),
&
\end{tikzcd}
$$
where $\mathcal R$ is the Berkovich retraction onto the skeleton. Moreover,
$
{\rm trop}_{\Sigma,\mathrm{rig}} \colon
\overline{\Sigma}\!\left(\overline{\mathcal{H}}^{{\rm an}}_{\mathrm{rig},0,\mathbb P^1,D}(\vec\mu^*)\right)
\to
\overline{\mathcal{H}}^{{\rm trop}}_{\mathrm{rig},0,\mathbb P^1,D}(\vec\mu^*)
$
is an isomorphism of extended cone complexes.
\end{thm}

\begin{proof}
By Lemmas \ref{lem_Xx0_prop} and \ref{lem_fiber_comp}, the chosen compactification $
\overline{\mathcal{H}}_{\mathrm{rig},0,\mathbb P^1,D}(\vec{\mu}^*)$
is a proper toroidal Deligne--Mumford stack. Hence, by Abramovich--Caporaso--Payne \cite{ACP15},
its analytification admits a canonical skeleton $
\overline{\Sigma}\!\left(\overline{\mathcal{H}}^{\rm an}_{\mathrm{rig},0,\mathbb P^1,D}(\vec{\mu}^*)\right)$.

We first define the tropicalization map. Let $x\in \overline{\mathcal{H}}^{\rm an}_{\mathrm{rig},0,\mathbb P^1,D}(\vec{\mu}^*)$
be an analytic point represented by a family over a valuation ring. Forgetting the rigidifying data
$x_0$ and the tangent vector at $x_\infty$, we obtain a 
framed decorated tree $(T,\delta,v)$ as in
Section \ref{sec_fhur}. The marked point $x_0$ determines
an additional marked leg $\ell_0$ on the induced source tropical admissible cover. If $x_0$
specializes to the special locus, it contributes the additional boundary parameter described in
Lemma \ref{lem_Xx0_prop}; this is recorded by $\epsilon=1$, while otherwise $\epsilon=0$. Finally, the chosen
compactification of the $\mathbf G_m$-bundle of non-zero tangent vectors at $x_\infty$ contributes one
additional toric parameter. Taking valuations of the node-smoothing parameters, together with the
valuation of the $x_0$-boundary parameter when $\epsilon=1$, and the valuation of the tangent vector
toric parameter, defines a map
$$
{\rm trop}_{\mathrm{rig}}\colon
\overline{\mathcal{H}}^{\rm an}_{\mathrm{rig},0,\mathbb P^1,D}(\vec{\mu}^*)
\to
\overline{\mathcal{H}}^{\rm trop}_{\mathrm{rig},0,\mathbb P^1,D}(\vec{\mu}^*).
$$

We next show that ${\rm trop}_{\mathrm{rig}}$ factors through the skeleton. \'Etale locally near a boundary
point of $\overline{\mathcal{H}}_{\mathrm{rig},0,\mathbb P^1,D}(\vec{\mu}^*)$, the completed local ring is obtained from
the local ring of the framed admissible cover space by adjoining the additional rigidifying
parameters of Section \ref{sec_3.2_rfhur}. Namely, smooth parameters for the free motion of $x_0$, in the boundary case
at most one additional monomial coordinate corresponding to the bubbling or specialization of
$x_0$, and one toric coordinate coming from the chosen compactification of the $\mathbf G_m$-bundle of
non-zero tangent vectors at $x_\infty$. Thus the characteristic monoid is exactly the monoid
generated by the node-smoothing parameters of the framed admissible cover degeneration, together
with the $x_0$-boundary parameter when present, and the tangent vector toric parameter. These are
precisely the coordinates of the cone $\overline{\sigma}^{\mathrm{rig}}_{\Xi}$ attached to the corresponding
rigidified framed decorated type $\Xi$. Hence ${\rm trop}_{\mathrm{rig}}$ factors through the canonical
retraction to the skeleton.

Finally, we identify the two extended cone complexes. A toroidal boundary stratum of
$\overline{\mathcal{H}}_{\mathrm{rig},0,\mathbb P^1,D}(\vec{\mu}^*)$
determines the framed decorated tree $(T,\delta,v)$ coming from the target dual tree and the admissible cover
monodromy data, the additional marked leg $\ell_0$ on the induced source combinatorial admissible
cover, and the indicator $\epsilon\in\{0,1\}$ recording whether $x_0$ contributes the extra
boundary modulus of Lemma \ref{lem_Xx0_prop}. The compactified tangent vector direction contributes the final
toric coordinate. Thus the toroidal boundary strata are indexed by the same rigidified framed
decorated types $\Xi$ that index the cones of
$
\overline{\mathcal{H}}^{\rm trop}_{\mathrm{rig},0,\mathbb P^1,D}(\vec{\mu}^*).
$
Under specialization, nodes are smoothed or created exactly by contracting edges of $T$, the
distinguished vertex is sent to its image, the marked leg $\ell_0$ is carried along under
stabilization of the source combinatorial admissible cover, and the additional $x_0$-boundary
direction is retained exactly in those specializations for which $x_0$ remains on the special
locus. Therefore the face morphisms agree on both sides.

It follows that
$$
\overline{\Sigma}\!\left(\overline{\mathcal{H}}^{\rm an}_{\mathrm{rig},0,\mathbb P^1,D}(\vec{\mu}^*)\right)
\cong
\overline{\mathcal{H}}^{\rm trop}_{\mathrm{rig},0,\mathbb P^1,D}(\vec{\mu}^*)
$$
as extended cone complexes. This completes the proof.
\end{proof}

\section{Comparing \texorpdfstring{$\overline{\mathcal{H}}_{{\rm rig},0,\mathbb{P}^1,D}^{\mathrm{trop}}(\Vec{\mu}^*)$}{Htrop} and space of trees}\label{sec:4-trees}

In this section, we define the topological space $\mathcal{T}_D(\Vec{\mu}^*)$ of {\it framed decorated polynomial trees} with ramification profile $\vec{\mu}^*$ and then prove the following theorem and Theorem \ref{thm:main-1}.

\begin{thm} \label{thm:comp-2}
For each polynomial ramification profile $\vec\mu^*$, there is a canonical isomorphism of extended cone complexes
$$
\Phi_{\vec\mu^*}\colon \overline{\mathcal{T}}_D(\vec\mu^*) \xrightarrow{\sim} \overline{\mathcal{H}}^{\rm trop}_{\mathrm{rig},0,\mathbb P^1,D}(\vec\mu^*).
$$
In particular, $\Phi_{\vec\mu^*}$ is a homeomorphism.
\end{thm}

\subsection{Framed decorated polynomial tree}

Before giving the formal definition, we explain the idea. Roughly speaking, a framed decorated polynomial tree is a polynomial tree with extra marking data remembering how it arises from a degeneration of rigidified framed polynomial covers. 
More specifically, the polynomial dynamics is recorded by an infinite rooted tree and its harmonic self-map, while the admissible cover degeneration is recoreded by the finite core of the tree. The decoration remembers the Hurwitz monodromy data, the framed vertex remembers the parametrized target component, the marked source leg remembers the rigidifying point $x_0$, and the remaining boundary parameters record the possible degeneration of $x_0$ and of the tangent vector rigidification at $x_\infty$.

\begin{defn}\label{def:framed_decorated_polynomial_tree}
A \emph{framed decorated polynomial tree of degree $D$} consists of the following data:
\begin{enumerate}
    \item an infinite rooted metric tree $T^\infty$ with distinguished root end $\infty_{T^\infty}$;

    \item a continuous harmonic map $F \colon T^\infty \to T^\infty$
    of degree $D$, preserving the root end and affine of integer slope on each edge;

    \item a finite connected rooted subtree $T_{\mathrm{core}} \subset T^\infty,$
    called the \emph{core};

    \item a rigidified framed decorated type
    $$
        \Xi=(T,\delta,v,\Gamma_{\rm src}^+,\ell_0,\epsilon,\rho)\in {\rm Stab}^{\mathrm{rig,fr}}_{0,B}(\vec{\mu}^{*}),
    $$
    such that the finite combinatorial data carried by $T_{\mathrm{core}}$ are identified with the source side combinatorial admissible cover induced by $\Xi$. More precisely,
    \begin{enumerate}
        \item the rooted combinatorial type of $T_{\mathrm{core}}$ agrees with the source combinatorial admissible cover $\Gamma_{\mathrm{src}} \to T$
        induced by $(T,\delta,v)$;

        \item the distinguished point $x_\infty$ lies on the root end of $T_{\mathrm{core}}$;

        \item the marked point $x_0$ is represented on $T_{\mathrm{core}}$ by the marked leg $\ell_0$;

        \item the chosen tangent direction at $x_\infty$ is the rooted flag pointing toward the root end;

        \item the decoration on $T_{\mathrm{core}}$ is the source side decoration induced from $(T,\delta,v)$;

        \item the indicator $\epsilon\in\{0,1\}$ records whether $x_0$ contributes the additional boundary modulus, exactly as in Definition \ref{def:rigidified_tropical_space}.
    \end{enumerate}

    \item the local tropical admissible cover Riemann--Hurwitz condition holds at every vertex of the finite core;

    \item the complement $T^\infty\setminus T_{\mathrm{core}}$ is obtained from the boundary data of the core by the polynomial tree completion (or inverse image completion) procedure (i.e., first take a direct limit to get an invariant tree, then apply \cite[Theorem 5.7]{DM08}), and carries no additional free moduli.
\end{enumerate}
\end{defn}

\begin{rmk}\label{rmk-4.3}
In Definition \ref{def:framed_decorated_polynomial_tree}, the definition of harmonic maps and local Riemann-Hurwitz condition are the same as in \cite{CMR16}. Moreover,
the finite core $T_{\mathrm{core}}$ is the only part carrying moduli; the infinite rooted tree $T^\infty$ is regarded as the polynomial tree completion (or inverse image completion) of this rigidified framed decorated core.
\end{rmk}

\begin{defn}
Two framed decorated polynomial trees are \emph{combinatorially isomorphic} if there is an isomorphism $\Phi\colon T^\infty\to (T')^\infty$ of the underlying rooted trees 
such that: 
\begin{enumerate} 
\item $\Phi$ sends the root, the end corresponding to infinity, and the finite core $T_{\rm core}$ to the corresponding root, end, and finite core of $ T'$; 
\item $\Phi$ conjugates the combinatorial dynamics, i.e., $\Phi\circ F=F'\circ \Phi$; 
\item $\Phi$ preserves local degrees on vertices and edges; 
\item the induced map on the finite core identifies the associated source side tropical admissible cover $ \Gamma_{\rm src}^+\to T$ with $ (\Gamma'_{\rm src})^+\to T'; $ 
\item under this identification, the Glynn decoration $\delta$ is carried to $\delta'$ up to framed Hurwitz equivalence, and the distinguished framed vertex is preserved, i.e., $ \Phi(v)=v'$;  \item the marked source leg and rigidifying data are preserved, i.e., $ \Phi(\ell_0)=\ell'_0,\epsilon=\epsilon',$ and $ \rho=\rho'$.
\end{enumerate}
\end{defn}

\begin{rmk}
A combinatorial isomorphism records the discrete framed decorated polynomial tree structure, and forgets the metric parameters.   
\end{rmk}

\begin{defn}
A framed decorated polynomial tree is said to be \emph{of profile $\vec\mu^*$} if the rigidified framed decorated type $\Xi=(T,\delta,v,\Gamma_{\rm src}^+,\ell_0,\epsilon,\rho)$
attached to its core lies in ${\rm Stab}^{\mathrm{rig,fr}}_{0,B}(\vec\mu^*)$.
\end{defn}

\subsection{The space of framed decorated polynomial trees}
Let $\mathcal{I}(\vec\mu^*)$ denote the set of combinatorial isomorphism classes of framed decorated polynomial trees of profile $\vec\mu^*$.
By Definition \ref{def:framed_decorated_polynomial_tree}, every framed decorated polynomial tree of profile $\vec\mu^*$ comes equipped with a finite core together with a rigidified framed decorated type
$$
\Xi=(T,\delta,v,\Gamma_{\rm src}^+,\ell_0,\epsilon,\rho)\in {\rm Stab}^{\mathrm{rig,fr}}_{0,B}(\vec\mu^*).
$$
Since combinatorial isomorphisms of framed decorated polynomial trees preserve the core and the full rigidified framed decorated type, this assignment descends to a well-defined map
$$
\mathcal{I}(\vec\mu^*) \to {\rm Stab}^{\mathrm{rig,fr}}_{0,B}(\vec\mu^*).
$$

\begin{lem}\label{lem:core_bijection}
The map $\mathcal{I}(\vec\mu^*) \to {\rm Stab}^{\mathrm{rig,fr}}_{0,B}(\vec\mu^*)$ is bijective.
\end{lem}

\begin{proof}
For surjectivity, let $\Xi=(T,\delta,v,\Gamma_{\rm src}^+,\ell_0,\epsilon,\rho)\in {\rm Stab}^{\mathrm{rig,fr}}_{0,B}(\vec\mu^*).$
The data $(T,\delta,v)$ determine the associated source side combinatorial admissible cover, and together with the marked leg $\ell_0$, the rooted flag at $x_\infty$, and the indicator $\epsilon$, they determine a finite rigidified framed decorated core of profile $\vec\mu^*$. By the completion procedure of Definition \ref{def:framed_decorated_polynomial_tree}, this core extends to a framed decorated polynomial tree of profile $\vec\mu^*$. Thus every element of ${\rm Stab}^{\mathrm{rig,fr}}_{0,B}(\vec\mu^*)$ arises from some framed decorated polynomial tree.

For injectivity, suppose two framed decorated polynomial trees of profile $\vec\mu^*$ determine the same rigidified framed decorated type $\Xi=(T,\delta,v,\Gamma_{\rm src}^+,\ell_0,\epsilon,\rho).$
Then their finite cores agree as rigidified framed decorated cores. Since, by Definition \ref{def:framed_decorated_polynomial_tree}, the complement of the core is obtained from the boundary data of the core by the polynomial tree completion procedure and carries no additional free moduli, the two framed decorated polynomial trees are isomorphic. Hence the map is injective.

Therefore the map is bijective.
\end{proof}

\begin{defn} \label{def:4.6}
Define
$$
\mathcal C(\vec\mu^*) :=  \mathcal{I}(\vec\mu^*)={\rm Stab}^{\mathrm{rig,fr}}_{0,B}(\vec\mu^*).
$$

For $\Xi\in \mathcal C(\vec\mu^*)$, define the associated extended cone
$$
\overline{\sigma}_\Xi := \overline{\sigma}^{\mathrm{rig}}_\Xi
=
(\mathbb R_{\ge 0}\cup\{\infty\})^{E(T)}
\times
(\mathbb R_{\ge 0}\cup\{\infty\})^\epsilon
\times
(\mathbb R_{\ge 0} \cup\{ \infty\}).
$$
The first factor records the edge-length parameters of the target tree $T$, the second factor records the additional boundary parameter contributed by $x_0$ when $\epsilon=1$, and the third factor records the toric parameter coming from the chosen compactification of the non-zero tangent vector bundle at $x_\infty$.

Each point of $\overline{\sigma}_\Xi$ determines a metric realization of the rigidified framed decorated core associated to $\Xi$, and hence, by the completion procedure of Definition \ref{def:framed_decorated_polynomial_tree}, a framed decorated polynomial tree, unique up to isomorphism.

If $\Xi'$ is obtained from $\Xi$ by a specialization of rigidified framed decorated types, then there is a natural face morphism
$$
j_{\Xi',\Xi} \colon \overline{\sigma}_{\Xi'}\hookrightarrow \overline{\sigma}_\Xi.
$$
\end{defn}

\begin{defn}\label{def:TD}
The space of framed decorated polynomial trees of profile $\vec\mu^*$ is the extended cone complex
$$
\overline{\mathcal{T}}_D(\vec\mu^*)
:=
\varinjlim_{\Xi\in \mathcal C(\vec\mu^*)}
(\overline{\sigma}_\Xi,j_{\Xi',\Xi}),
$$
where the transition maps are the face morphisms induced by specialization of rigidified framed decorated types.
We equip $\overline{\mathcal{T}}_D(\vec\mu^*)$ with the resulting colimit topology, and call it the {\it cone-complex topology}.
\end{defn}

Now we define a {\it geometric topology} on $\overline{\mathcal{T}}_D(\vec\mu^*)$.

\begin{defn}\label{def:geometric_topology}
Let $[T_n],[T]\in \overline{\mathcal{T}}_D(\vec\mu^*)$. We say that $[T_n]$ \emph{converges geometrically} to $[T]$ if there exists a rigidified framed decorated type
$$
\Xi=(T,\delta,v,\Gamma_{\rm src}^+,\ell_0,\epsilon,\rho)\in {\rm Stab}^{\mathrm{rig,fr}}_{0,B}(\vec\mu^*)
$$
and, for all sufficiently large $n$, realizations of $[T_n]$ and $[T]$ by points in a common cone $\overline{\sigma}_\Xi$ or in a face of $\overline{\sigma}_\Xi$, such that

\begin{enumerate}
\item the discrete rigidified framed decorated data are constant on the non-contracted part for all sufficiently large $n$;

\item the coordinates corresponding to edges contracted in passing from the realization of $[T_n]$ to that of $[T]$ tend to $0$;

\item the remaining cone coordinates converge to the corresponding coordinates of $[T]$;

\item the completion outside the finite core is the polynomial tree completion of Definition \ref{def:framed_decorated_polynomial_tree}, so contributes no additional free parameters.
\end{enumerate}
\end{defn}

\begin{rmk}
The geometric topology is an analogue in the current setting of the geometric topology considered by DeMarco--McMullen \cite{DM08}. Since the infinite rooted completion carries no additional moduli beyond those of the finite rigidified framed decorated core, this topology is entirely controlled by convergence of the finite core data together with the collapse of edges whose lengths tend to $0$.
\end{rmk}

\begin{prop}\label{prop:2-top-agree}
The geometric topology on $\overline{\mathcal{T}}_D(\vec\mu^*)$ agrees with the cone-complex topology.
\end{prop}

\begin{proof}
We first show that convergence in the cone-complex topology implies geometric convergence.
Suppose $[T_n]\to [T]$
in the cone-complex topology on $\overline{\mathcal{T}}_D(\vec\mu^*)$.
By definition of the colimit topology, after passing to a subsequence, there exists a rigidified framed decorated type $\Xi\in {\rm Stab}^{\mathrm{rig,fr}}_{0,B}(\vec\mu^*)$
such that each $[T_n]$ is represented by a point $x_n\in \overline{\sigma}_\Xi$, while $[T]$ is represented by a point $x$ in a face $\overline{\sigma}_{\Xi'}\subset \overline{\sigma}_\Xi$, where $\Xi'$ is obtained from $\Xi$ by specialization.

The specialization $\Xi\rightsquigarrow \Xi'$ contracts exactly those target edges whose coordinates tend to $0$, while preserving the rigidified framed decorated data on the non-contracted part. Since the remaining coordinates converge to the corresponding coordinates of $x$, the associated finite rigidified framed decorated cores converge geometrically. By Definition \ref{def:framed_decorated_polynomial_tree} and Remark \ref{rmk-4.3}, the infinite rooted completion contributes no further moduli. Hence $[T_n]\to [T]$ geometrically.

Conversely, suppose $[T_n]\to [T]$ geometrically. Then, after passing to a subsequence, the finite rigidified framed decorated core data stabilize to a common rigidified framed decorated type $\Xi$, and the limiting object is obtained from $\Xi$ by contracting exactly those edges whose lengths tend to $0$. Thus $[T_n]$ is represented by points $x_n\in \overline{\sigma}_\Xi$, while $[T]$ is represented by a point $x$ in the corresponding face $\overline{\sigma}_{\Xi'}\subset \overline{\sigma}_\Xi$. The geometric convergence conditions are precisely that the coordinates of $x_n$ converge to those of $x$ in the Euclidean topology on $\overline{\sigma}_\Xi$. Therefore $[T_n]\to [T]$ in the cone-complex topology. This completes the proof.
\end{proof}

\subsection{Proof of Theorems \ref{thm:comp-2} and \ref{thm:main-1}}

\begin{proof}[Proof of Theorem \ref{thm:comp-2}]
By Definition \ref{def:4.6}, the cones of $\overline{\mathcal{T}}_D(\vec\mu^*)$ are indexed by
$
\mathcal C(\vec\mu^*)= \mathcal{I}(\vec\mu^*) = {\rm Stab}^{\mathrm{rig,fr}}_{0,B}(\vec\mu^*).
$
By Definition \ref{def:rigidified_tropical_space}, the cones of $\overline{\mathcal{H}}^{\rm trop}_{\mathrm{rig},0,\mathbb P^1,D}(\vec\mu^*)$ are indexed by the same set.

For each $\Xi \in {\rm Stab}^{\mathrm{rig,fr}}_{0,B}(\vec\mu^*)$, the associated cone on the $\overline{\mathcal{T}}_D(\vec\mu^*)$-side is, by Definition \ref{def:4.6},
$$
\overline{\sigma}_\Xi
=
(\mathbb R_{\ge 0}\cup\{\infty\})^{E(T)}
\times
(\mathbb R_{\ge 0}\cup\{\infty\})^\epsilon
\times
(\mathbb R_{\ge 0}\cup\{\infty\}),
$$
while the associated cone on the $\overline{\mathcal{H}}^{\rm trop}_{\mathrm{rig},0,\mathbb P^1,D}(\vec\mu^*)$-side is, by Definition \ref{def:rigidified_tropical_space},
$$
\overline{\sigma}^{\mathrm{rig}}_\Xi
=
(\mathbb R_{\ge 0}\cup\{\infty\})^{E(T)}
\times
(\mathbb R_{\ge 0}\cup\{\infty\})^\epsilon
\times
(\mathbb R_{\ge 0}\cup\{\infty\}).
$$
Thus the cones on both sides agree canonically for every $\Xi$.

Now let $\Xi'$ be obtained from $\Xi$ by a specialization of rigidified framed decorated types. On the tropical Hurwitz side, Definition \ref{def:rigidified_tropical_space} associates to this specialization the face morphism
$$
j^{\mathrm{rig}}_{\Xi',\Xi}\colon \overline{\sigma}^{\mathrm{rig}}_{\Xi'}\hookrightarrow \overline{\sigma}^{\mathrm{rig}}_\Xi.
$$
On the polynomial tree side, Definition \ref{def:4.6} associates to the same specialization the face morphism
$$
j_{\Xi',\Xi}\colon \overline{\sigma}_{\Xi'}\hookrightarrow \overline{\sigma}_\Xi.
$$
Since both are induced by the same specialization of the same rigidified framed decorated type $\Xi$, these face morphisms agree under the canonical identification $\overline{\sigma}_\Xi=\overline{\sigma}^{\mathrm{rig}}_\Xi$.

Therefore the two extended cone complexes are colimits of the same system of cones with the same transition maps
$$
\overline{\mathcal{T}}_D(\vec\mu^*)
=
\varinjlim_{\Xi\in {\rm Stab}^{\mathrm{rig,fr}}_{0,B}(\vec\mu^*)}
(\overline{\sigma}_\Xi,j_{\Xi',\Xi}),
$$
and
$$
\overline{\mathcal{H}}^{\rm trop}_{\mathrm{rig},0,\mathbb P^1,D}(\vec\mu^*)
=
\varinjlim_{\Xi\in {\rm Stab}^{\mathrm{rig,fr}}_{0,B}(\vec\mu^*)}
(\overline{\sigma}^{\mathrm{rig}}_\Xi,j^{\mathrm{rig}}_{\Xi',\Xi}).
$$
Hence there is a canonical isomorphism of extended cone complexes
$
\Phi_{\vec\mu^*} \colon \overline{\mathcal{T}}_D(\vec\mu^*)\xrightarrow{\sim}
\overline{\mathcal{H}}^{\rm trop}_{\mathrm{rig},0,\mathbb P^1,D}(\vec\mu^*).
$
In particular, $\Phi_{\vec\mu^*}$ is a homeomorphism with respect to the cone-complex topologies. This completes the proof.
\end{proof}

\begin{proof}[Proof of Theorem \ref{thm:main-1}]
The conclusion follows from Theorem \ref{thm:comp-2} and Theorem \ref{thm_rig_framed_CMR16}.   
\end{proof}

\section{Proof of Theorem \ref{thm:main-compactification}} \label{sec:pf-1.2}

In this section, we prove Theorem \ref{thm:main-compactification}. The proof adapts the general framework of DeMarco--McMullen's compactification
of polynomial moduli by projectivized space of polynomial trees; see \cite[Sections 5-8]{DM08}.

\subsection{Projectivized framed decorated tree space}
Recall that $\mathcal{T}_D(\vec\mu^*)$ is an extended cone complex whose cones
are indexed by rigidified framed decorated types. On each cone
$$
\sigma_\Xi =
(\mathbb R_{\ge 0})^{E(T)}
\times
(\mathbb R_{\ge 0})^\epsilon
\times
\mathbb R_{\ge 0},
$$
the group $\mathbb R_{>0}$ acts by simultaneous rescaling of all
length coordinates. We define
$$
\mathbb{P}\mathcal{T}_D(\vec\mu^*) :=
\bigl(\mathcal{T}_D(\vec\mu^*)\setminus \{0\}\bigr)/\mathbb R_{>0}.
$$
In other words, $\mathbb{P}\mathcal{T}_D(\vec\mu^*)$ is the space of framed decorated
polynomial trees up to global rescaling of the metric and boundary
length data.

\begin{lem}\label{lem:5.1}
The space $\mathbb{P}\mathcal{T}_D(\vec\mu^*)$ is compact.
\end{lem}

\begin{proof}
For fixed degree $D$ and fixed profile $\vec\mu^*$, there are only
finitely many rigidified framed decorated combinatorial types. For each
type $\Xi$, the projectivization of the finite cone $\sigma_\Xi\setminus\{0\}$
is a compact simplex, possibly with some faces corresponding to
specializations of $\Xi$. The face maps are induced by contractions of
edges and by the corresponding specializations of the rigidified framed
decorated data. Therefore $\mathbb{P}\mathcal{T}_D(\vec\mu^*)$ is obtained by gluing
finitely many compact simplices along closed faces. Hence it is compact.
\end{proof}

\subsection{The tropicalization map and its projectivization}\label{sec:tau-map}

We next define and study the analogue of the quotient tree map $\tau \colon \mathrm{MPoly}_D^*\to \mathcal T_D^{\rm DM}$ of DeMarco--McMullen \cite[Section 6]{DM08}. 

Let $\overline {\mathcal H}_{\mathrm{rig},0,\mathbb P^1,D}(\vec\mu^*)$
be the proper toroidal compactification constructed in Section \ref{sec_rfhur}.
Recall that
$$
\mathcal{H}_{\mathrm{rig},0,\mathbb P^1,D}(\vec\mu^*)
\cong
{\rm Poly}_D(\vec\mu^*)
$$
by Remark \ref{rmk_polyD}. We denote its boundary by
$$
\partial \overline {\mathcal H}_{\mathrm{rig},0,\mathbb P^1,D}(\vec\mu^*)
:=
\overline {\mathcal H}_{\mathrm{rig},0,\mathbb P^1,D}(\vec\mu^*)
\setminus
\mathcal{H}_{\mathrm{rig},0,\mathbb P^1,D}(\vec\mu^*).
$$ 
Choose a sufficiently small analytic neighborhood $U$ of the boundary, and set 
$$ 
{\rm Poly}_D(\vec\mu^*)^{\mathrm{deg}} := U\cap {\rm Poly}_D(\vec\mu^*). $$ 

We first define the map locally in a toroidal chart. Suppose that a boundary stratum is indexed by a rigidified framed decorated type $\Xi=(T,\delta,v,\Gamma_{\rm src}^+,\ell_0,\epsilon,\rho), $ and let $q_1,\ldots,q_N$ be toroidal boundary coordinates in a chart centered at this stratum. These coordinates record the smoothing parameters of the target nodes, the possible boundary parameter contributed by $x_0$, and the toric parameter coming from the chosen compactification of the non-zero tangent vector bundle at $x_\infty$. The associated cone is 
$$\sigma_\Xi\cong \mathbb R_{\ge 0}^N.$$ 
For a point $f\in {\rm Poly}_D(\vec\mu^*)^{\mathrm{deg}} $ lying in this chart, define $$ \tau_{\vec\mu^*,\Xi}(f) := \bigl( -\log |q_1(f)|,\ldots,-\log |q_N(f)| \bigr) \in \sigma_\Xi. $$ Together with the label $\Xi$, this gives a point of the cone
$\sigma_\Xi$ in the skeleton of the toroidal compactification, by definition of the skeleton of a toroidal embedding. Under the identifications of Theorems \ref{thm_rig_framed_CMR16} and \ref{thm:comp-2}, this cone is
identified with the cone of $\mathcal T_D(\vec\mu^*)$ indexed by the
same rigidified framed decorated type $\Xi$.

Thus, locally near the boundary, we obtain a {\it logarithmic tropicalization map} 
$$ \tau_{\vec\mu^*,\Xi} \colon {\rm Poly}_D(\vec\mu^*)^{\mathrm{deg}}\cap U_\Xi \to \sigma_\Xi \subset \mathcal T_D(\vec\mu^*). $$ The map $\tau_{\vec\mu^*,\Xi}$ depends on the choice of toroidal chart up to bounded error and the integral linear transition maps of the cone complex. More precisely, if $ q'_1,\ldots,q'_M$ is another system of toroidal boundary coordinates on an overlapping chart, then, after possibly refining the charts, the two systems are related by monomial transformations multiplied by units: $$ q_i' = u_i\prod_j q_j^{a_{ij}}, $$ where each $u_i$ is invertible along the boundary. Hence $$ -\log |q_i'(f)| = \sum_j a_{ij}(-\log |q_j(f)|)-\log |u_i(f)|. $$ The monomial part is exactly the integral linear transition map of the toroidal cone complex, while the unit term is bounded near the boundary. Thus the unprojectivized logarithmic vector is well-defined only up to bounded error in overlapping charts. For this reason, the canonical compactifying object is the {\it projectivization} of this logarithmic vector. 

\smallskip

We define the {\it local projectivized tropicalization map} by 
$$ \mathbb{P} \tau_{\vec\mu^*,\Xi}(f) := [ -\log |q_1(f)|:\cdots:-\log |q_N(f)| ] \in \mathbb P\sigma_\Xi \subset \mathbb P\mathcal T_D(\vec\mu^*). $$ Equivalently, $\mathbb{P} \tau_{\vec\mu^*,\Xi}(f)$ records the rigidified framed decorated type $\Xi$ together with the relative rates at which the toroidal boundary parameters vanish. 

Although $\mathbb P\tau_{\vec\mu^*,\Xi}$ is defined using a chart, its limits along {\it sequences} approaching the boundary are intrinsic. Indeed, if $f_n\in {\rm Poly}_D(\vec\mu^*)$ is a sequence converging to the toroidal boundary, then $$ \left\| \tau_{\vec\mu^*,\Xi}(f_n) \right\| \to \infty. $$ Therefore the bounded unit terms appearing in coordinate changes vanish after projectivization. It follows that if the projective logarithmic vectors 
$$\mathbb{P} \tau_{\vec\mu^*,\Xi}(f_n) = [ -\log |q_1(f_n)|:\cdots:-\log |q_N(f_n)| ] $$ converge in one toroidal chart, then the same limiting point of $\mathbb P\mathcal T_D(\vec\mu^*)$ is obtained in any other toroidal chart. 

\begin{defn}
We say that a sequence $\{f_n\} \subset {\rm Poly}_D(\vec\mu^*)$ {\it has projectivized tropical limit} $[\tau]\in \mathbb P\mathcal T_D(\vec\mu^*)$ if, after passing to a subsequence and choosing a toroidal chart near the limiting boundary stratum, the projective logarithmic vectors $\mathbb{P} \tau_{\vec\mu^*,\Xi}(f_n)$ converge to the point of $\mathbb P\sigma_\Xi$ representing $[\tau]$.  
\end{defn}
By the preceding paragraph, this definition is independent of the chosen toroidal chart. 

\begin{rmk}\label{rmk:degen-fam}
Degenerating one-parameter families give a useful special case of this definition. Let $ f\colon \Delta^*\to {\rm Poly}_D(\vec\mu^*)$ be a degenerating holomorphic family, where $\Delta^*$ is the punctured unit disk in $\mathbb C$. Since $\overline {\mathcal H}_{\mathrm{rig},0,\mathbb P^1,D}(\vec\mu^*)$ is proper, after a finite base change the family extends to a holomorphic map $\bar f\colon \Delta\to \overline {\mathcal H}_{\mathrm{rig},0,\mathbb P^1,D}(\vec\mu^*).$ Assume that $\bar f(0)$ lies in a toroidal chart with boundary coordinates $q_1,\ldots,q_N$. Then $$ -\log |q_i(f(t))| = \operatorname{ord}_0(q_i\circ \bar f)\cdot (-\log |t|) +O(1) $$ as $t\to 0$. Consequently the projectivized tropical limit of the family is $$ [ \operatorname{ord}_0(q_1\circ \bar f): \cdots: \operatorname{ord}_0(q_N\circ \bar f) ]. $$ Thus degenerating arcs produce rational projective points of the cone, whereas arbitrary degenerating sequences may converge to arbitrary real points of the projectivized skeleton. 
\end{rmk}

We next prove the properness and surjectivity properties needed for compactification. 
\begin{lem}\label{lem:5.3}
A sequence
$\{f_n\}\subset {\rm Poly}_D(\vec\mu^*)$
leaves every compact subset of ${\rm Poly}_D(\vec\mu^*)$ if and only if,
after passing to a subsequence, it enters 
${\rm Poly}_D(\vec\mu^*)^{\mathrm{deg}}$ and 
$\max_i\{-\log |q_i(f_n)|\}\to \infty.
$
Equivalently, divergence in ${\rm Poly}_D(\vec\mu^*)$ is the same as
approaching the toroidal boundary of
$\overline {\mathcal H}_{\mathrm{rig},0,\mathbb P^1,D}(\vec\mu^*).
$
\end{lem}

\begin{proof}
Since $\overline {\mathcal H}_{\mathrm{rig},0,\mathbb P^1,D}(\vec\mu^*)$
is proper, every sequence in ${\rm Poly}_D(\vec\mu^*)$ has a subsequence
converging in the compactification. If the sequence does not leave every
compact subset of ${\rm Poly}_D(\vec\mu^*)$, then some subsequence converges
to a point of the open stratum
$$
\mathcal{H}_{\mathrm{rig},0,\mathbb P^1,D}(\vec\mu^*)
\cong
{\rm Poly}_D(\vec\mu^*).
$$
Conversely, if a sequence leaves every compact subset of the open
stratum, then every convergent subsequence in the proper compactification
has limit in the boundary. In toroidal coordinates near such a boundary
point, approaching the boundary is precisely the condition that at least
one boundary coordinate $q_i$ tends to $0$, or equivalently that
$$
\max_i\{-\log |q_i|\}\to \infty.
$$
This completes the proof.
\end{proof}

\begin{cor}\label{cor:5.4}
Every sequence
$f_n\in {\rm Poly}_D(\vec\mu^*)
$ leaving every compact subsets of ${\rm Poly}_D(\vec\mu^*)$ has a projectivized tropical limit.
\end{cor}

\begin{proof}
By Lemma \ref{lem:5.3}, after passing to a subsequence, $f_n$ lies in
the degenerating locus and $\max_i\{-\log |q_i|\}\to \infty$. The
images
$[\mathbb{P} \tau_{\vec\mu^*,\Xi}(f_n)]
$
lie in the space $\mathbb{P}\mathcal{T}_D(\vec\mu^*)$, which is compact by Lemma \ref{lem:5.1}. Hence they
have a convergent subsequence.
\end{proof}

\begin{prop}\label{prop:real-boundary-points}
For every point
$[\tau]\in \mathbb{P}\mathcal{T}_D(\vec\mu^*)$,
there exists a sequence
$\{f_n\} \subset {\rm Poly}_D(\vec\mu^*)
$
leaving every compact subset of ${\rm Poly}_D(\vec\mu^*)$ such that
$\{f_n\}$ has projectivized tropical limit $[\tau]$.
\end{prop}

\begin{proof}
Let $[\tau]\in \mathbb{P}\mathcal{T}_D(\vec\mu^*)$
be represented by a point of the projectivization of a cone
$$
\sigma_\Xi =
(\mathbb R_{\ge 0})^{E(T)}
\times
(\mathbb R_{\ge 0})^{\epsilon}
\times
\mathbb R_{\ge 0},
$$
where
$
\Xi=(T,\delta,v,\Gamma_{\rm src}^+,\ell_0,\epsilon,\rho)
$
is a rigidified framed decorated type. Choose a representative
$
a=(a_1,\ldots,a_N)\in \mathbb R_{\ge 0}^N\setminus\{0\}
$
of the projective length vector defining $[\tau]$.

We now interpret this cone as a toroidal cone of the compactification.
By Theorem \ref{thm:comp-2}, the cone complex $\mathcal{T}_D(\vec\mu^*)$ is canonically
identified with the rigidified framed tropical Hurwitz space
$
\mathcal{H}^{\rm trop}_{\mathrm{rig},0,\mathbb P^1,D}(\vec\mu^*).
$
By Theorem \ref{thm_rig_framed_CMR16}, this tropical Hurwitz space is canonically
identified with the Berkovich skeleton
$
\Sigma\!\left(
\mathcal{H}^{\rm an}_{\mathrm{rig},0,\mathbb P^1,D}(\vec\mu^*)
\right)
$
of the proper toroidal compactification
$
\overline {\mathcal H}_{\mathrm{rig},0,\mathbb P^1,D}(\vec\mu^*).
$
Therefore the cone $\sigma_\Xi$ corresponds to a toroidal boundary
stratum of
$
\overline {\mathcal H}_{\mathrm{rig},0,\mathbb P^1,D}(\vec\mu^*).
$
Choose a point $x$ in this boundary stratum and choose a toroidal
chart centered at $x$. 
In this chart the boundary is locally described by
toroidal coordinates
$
q_1,\ldots,q_N
$
together with smooth coordinates along the stratum. The coordinates
$q_i$ correspond to the smoothing parameters of the target nodes, the
possible $x_0$-boundary parameter, and the toric tangent vector
parameter at $x_\infty$.

We first assume that all $a_i$ are rational. After multiplying the
vector $a$ by a common positive integer, we may assume that
$
a_i\in \mathbb Z_{\ge 0}.
$
Choose a general point of the boundary stratum and define a holomorphic
arc in the toroidal chart by
$
q_i(t)=t^{a_i}
$
for the boundary coordinates with $a_i>0$, by taking $q_i(t)$ to be
a non-zero constant for those $a_i=0$, and by choosing the smooth
coordinates along the stratum to be general constants. The family $f_t$ is a degenerating family as in Remark \ref{rmk:degen-fam}. For $t\neq 0$
sufficiently small, the point lies in the open stratum
$$
\mathcal{H}_{\mathrm{rig},0,\mathbb P^1,D}(\vec\mu^*)
\cong
{\rm Poly}_D(\vec\mu^*).
$$
The associated one-parameter family therefore gives polynomials
$f_t\in {\rm Poly}_D(\vec\mu^*)$, and its tropicalization has valuation
vector proportional to $a$. Hence its projectivized tropicalization is
exactly $[\tau]$.

For a general real vector
$
a\in \mathbb R_{\ge 0}^N\setminus\{0\},
$
choose rational vectors
$$
a^{(k)}\in \mathbb Q_{\ge 0}^N\setminus\{0\}
$$
whose projective classes converge to the projective class of $a$.
By the rational case, for each $k$ there exists a degenerating
one-parameter family in ${\rm Poly}_D(\vec\mu^*)$ whose projectivized
tropicalization is represented by $a^{(k)}$. Choosing one sufficiently
far out point on the $k$-th family gives a sequence
$f_k\in {\rm Poly}_D(\vec\mu^*)$
which leaves every compact subset and whose projectivized tropical
limits converge to $[\tau]$.

It remains only to note that this construction is compatible with the
face identifications of $\mathbb{P}\mathcal{T}_D(\vec\mu^*)$. Under the identifications of
Theorems \ref{thm_rig_framed_CMR16} and \ref{thm:comp-2}, the face maps of $\mathcal{T}_D(\vec\mu^*)$ are
exactly the specialization maps of toroidal boundary strata in
$\overline {\mathcal H}_{\mathrm{rig},0,\mathbb P^1,D}(\vec\mu^*).
$
Thus, if some coordinates $a_i$ vanish, the corresponding point lies
on the appropriate face of the cone, which geometrically corresponds to
the associated contraction or specialization of the rigidified framed
decorated type. Hence the constructed sequence realizes the prescribed
point of the projectivized cone complex, not merely a point in a chosen
local chart.

Therefore every point of $\mathbb{P}\mathcal{T}_D(\vec\mu^*)$ is the projectivized
tropical limit of a divergent sequence in ${\rm Poly}_D(\vec\mu^*)$.
\end{proof}

\subsection{Compactification}
We define the set
$$
{\rm Poly}_D(\vec\mu^*)^T
:=
{\rm Poly}_D(\vec\mu^*)\sqcup \mathbb{P}\mathcal{T}_D(\vec\mu^*).
$$
We endow this set with the topology characterized as follows.
The restriction of the topology to the open set
${\rm Poly}_D(\vec\mu^*)$
is its usual complex analytic topology. The restriction of the
topology to the boundary
$\mathbb{P}\mathcal{T}_D(\vec\mu^*)$
is the projectivized cone-complex topology. Finally, a sequence
$\{f_n\} \subset {\rm Poly}_D(\vec\mu^*)$
converges to a boundary point
$[\tau]\in \mathbb{P}\mathcal{T}_D(\vec\mu^*)
$
if and only if $\{f_n\}$ leaves every compact subset of
${\rm Poly}_D(\vec\mu^*)$ and has projectivized tropical limit
$[\tau]$ in the sense of Section \ref{sec:tau-map}.
Equivalently, after passing to a toroidal chart near the limiting
boundary stratum of
$\overline {\mathcal H}_{\mathrm{rig},0,\mathbb P^1,D}(\vec\mu^*)$,
with boundary coordinates
$q_1,\ldots,q_N$,
we have the convergence
$f_n\to [\tau]$
if
$$
\max_i\{-\log |q_i(f_n)|\}\to \infty
$$
and the projective logarithmic vectors
$$
[-\log |q_1(f_n)|:\cdots:-\log |q_N(f_n)|]
$$
converge to the projective length vector representing $[\tau]$, with
the corresponding rigidified framed decorated type stabilizing after
passing to a subsequence. As noted in Section \ref{sec:tau-map}, this condition is independent
of the choice of toroidal coordinates.

\begin{prop} \label{prop:compactification}
The space
${\rm Poly}_D(\vec\mu^*)^T
:=
{\rm Poly}_D(\vec\mu^*)\sqcup \mathbb{P}\mathcal{T}_D(\vec\mu^*)
$
is compact.
\end{prop}

\begin{proof}
Choose a finite collection of toroidal charts
$U_\alpha$
covering the boundary
$\partial \overline {\mathcal H}_{\mathrm{rig},0,\mathbb P^1,D}(\vec\mu^*)$,
which is possible as the compactification is proper. In each chart
$U_\alpha$, let
$q_{\alpha,1},\ldots,q_{\alpha,N_\alpha}
$
be toroidal boundary coordinates. The corresponding local boundary
directions are parametrized by the projectivized cone
$\mathbb P(\mathbb R_{\ge 0}^{N_\alpha}\setminus\{0\}).
$
Under the identifications of Theorems \ref{thm_rig_framed_CMR16} and \ref{thm:comp-2}, these local
projectivized cones glue along their faces to the global projectivized
tree space $\mathbb{P}\mathcal{T}_D(\vec\mu^*)$.

Let $K$ be a compact subset of the open stratum whose complement is
contained in the union of these toroidal neighborhoods:
$$
{\rm Poly}_D(\vec\mu^*)\setminus K
\subset
\bigcup_\alpha (U_\alpha\cap {\rm Poly}_D(\vec\mu^*)).
$$
Such a $K$ exists as the complement of the chosen boundary
neighborhood is compact in the proper compactification.

It is enough to prove that every open cover of
${\rm Poly}_D(\vec\mu^*)^T
$
has a finite subcover. Let
$\mathcal U$
be an open cover of ${\rm Poly}_D(\vec\mu^*)^T$. Since
$\mathbb{P}\mathcal{T}_D(\vec\mu^*)$ is compact by Lemma \ref{lem:5.1}, finitely many elements
$V_1,\ldots,V_m\in \mathcal U
$
cover the boundary $\mathbb{P}\mathcal{T}_D(\vec\mu^*)$.

By the definition of the topology on ${\rm Poly}_D(\vec\mu^*)^T$, each
$V_j$ contains, together with its boundary part, all sufficiently
large points in the corresponding toroidal sectors whose projectivized
logarithmic valuation classes lie in that boundary part. Equivalently,
after shrinking to finitely many toroidal charts, there exists a number
$R>0$ such that the finite union
$V_1\cup\cdots\cup V_m$
contains every point $f\in {\rm Poly}_D(\vec\mu^*)$ satisfying
$$
\max_i\{-\log |q_{\alpha,i}(f)|\}>R
$$
in one of the boundary charts $U_\alpha$. In other words, the finite
union $V_1\cup\cdots\cup V_m$ contains the complement of a compact
subset of ${\rm Poly}_D(\vec\mu^*)$.

Let $K_R\subset {\rm Poly}_D(\vec\mu^*)$ denote the remaining part of the
open stratum, namely the set of points which either lie outside the
chosen boundary neighborhoods or have toroidal scale at most $R$ in
the relevant boundary chart. The set $K_R$ is compact, since its
closure in the proper compactification avoids the boundary. Therefore
$K_R$ is covered by finitely many further elements
$V_{m+1},\ldots,V_{m+\ell}\in \mathcal U.
$
Thus
$$
V_1,\ldots,V_m,V_{m+1},\ldots,V_{m+\ell}
$$
cover all of ${\rm Poly}_D(\vec\mu^*)^T$. Hence every open cover has a
finite subcover, and so ${\rm Poly}_D(\vec\mu^*)^T$ is compact.
\end{proof}

\begin{prop} \label{prop:inclusion}
The inclusion
$$
{\rm Poly}_D(\vec\mu^*)\hookrightarrow {\rm Poly}_D(\vec\mu^*)^T
$$
identifies ${\rm Poly}_D(\vec\mu^*)$ with a dense open subset of
${\rm Poly}_D(\vec\mu^*)^T$, and the boundary is naturally homeomorphic to
$\mathbb{P}\mathcal{T}_D(\vec\mu^*)$.
\end{prop}

\begin{proof}
By construction, the topology on the open part
${\rm Poly}_D(\vec\mu^*)
$
is its usual topology, and no point of $\mathbb{P}\mathcal{T}_D(\vec\mu^*)$ lies in the
open stratum. Hence ${\rm Poly}_D(\vec\mu^*)$ is an open subset of
${\rm Poly}_D(\vec\mu^*)^T$, and the complement is set-wise
$\mathbb{P}\mathcal{T}_D(\vec\mu^*)$.
The topology induced on this complement is, by definition, the
projectivized cone complex topology. Thus the boundary is naturally
homeomorphic to $\mathbb{P}\mathcal{T}_D(\vec\mu^*)$.

It remains to prove density of the open stratum. Let
$[\tau]\in \mathbb{P}\mathcal{T}_D(\vec\mu^*)
$
be any boundary point. By Proposition \ref{prop:real-boundary-points}, there exists a sequence
$\{f_n\}\subset  {\rm Poly}_D(\vec\mu^*)
$
leaving every compact subset of ${\rm Poly}_D(\vec\mu^*)$ whose
projectivized tropical limit is $[\tau]$. By the definition of the
topology on ${\rm Poly}_D(\vec\mu^*)^T$, this means that
$f_n\to [\tau]$.
Therefore every boundary point lies in the closure of
${\rm Poly}_D(\vec\mu^*)$. Hence ${\rm Poly}_D(\vec\mu^*)$ is dense in
${\rm Poly}_D(\vec\mu^*)^T$.
\end{proof}

\subsection{Proof of Theorem \ref{thm:main-compactification}} 
\begin{proof}[Proof of Theorem \ref{thm:main-compactification}]
By Propositions \ref{prop:compactification} and \ref{prop:inclusion}, the space
$$
{\rm Poly}_D(\vec\mu^*)^T
=
{\rm Poly}_D(\vec\mu^*)\sqcup \mathbb{P}\mathcal{T}_D(\vec\mu^*)
$$
is compact, contains ${\rm Poly}_D(\vec\mu^*)$ as a dense open subset, and
has boundary naturally homeomorphic to $\mathbb{P}\mathcal{T}_D(\vec\mu^*)$. Thus
$\mathbb{P}\mathcal{T}_D(\vec\mu^*)$ gives a compactification of 
${\rm Poly}_D(\vec\mu^*)$, and the proof is complete.
\end{proof}

\section{Proof of Theorem \ref{thm:main-relation-to-DM}}\label{sec:pf-1.3}
We compare the decorated tree compactification constructed in Section \ref{sec:pf-1.2}
with the DeMarco--McMullen compactification of polynomial moduli. We first note that our compactification is a compactification of the
{\it parameter stratum} ${\rm Poly}_D(\vec\mu^*)$,
whereas the DeMarco--McMullen compactification is a compactification of
the {\it moduli space} of polynomials. As a result, our boundary contains
additional directions coming from the rigidification, namely the
$x_0$-boundary parameter and the tangent vector compactification at
$x_\infty$. These directions need not change the underlying
DeMarco--McMullen polynomial tree. Thus we first restrict to the
part of the boundary where the underlying DeMarco--McMullen tree is
non-zero.

Let
$$
\Xi=(T,\delta,v,\Gamma_{\rm src}^+,\ell_0,\epsilon,\rho)
\in
{\rm Stab}^{\mathrm{rig},\mathrm{fr}}_{0,B}(\vec\mu^*)
$$
be a rigidified framed decorated type. The corresponding cone of
$\mathcal T_D(\vec\mu^*)$ is
$$
\sigma_\Xi
=
(\mathbb R_{\ge 0})^{E(T)}
\times
(\mathbb R_{\ge 0})^\epsilon
\times
\mathbb R_{\ge 0}.
$$
The first factor records the dynamical edge length data of the target tree. The remaining factors record rigidifying boundary data, i.e., the possible $x_0$-boundary parameter and the toric parameter coming from the compactification of the non-zero tangent vector bundle at $x_\infty$.

Define
$\sigma_\Xi^{\mathrm{dyn}}
\subset
\sigma_\Xi$
to be the complement of the locus on which all
dynamical edge length coordinates vanish and only the rigidifying
coordinates may be non-zero. These subsets are compatible with face maps in the sense that a face map
sends points with non-zero underlying DeMarco--McMullen length data to
points with non-zero underlying DeMarco--McMullen length data whenever
the corresponding forgetful map is defined.
Therefore, they define a subspace
$$
\mathcal T_D(\vec\mu^*)^{\mathrm{dyn}}
\subset
\mathcal T_D(\vec\mu^*).
$$
After projectivizing, we obtain
$\mathbb P\mathcal T_D(\vec\mu^*)^{\mathrm{dyn}}
\subset
\mathbb P\mathcal T_D(\vec\mu^*)$.

On each cone $\sigma_\Xi^{\mathrm{dyn}}$, define 
$$
\Psi_\Xi \colon 
\sigma_\Xi^{\mathrm{dyn}}
\to
\mathcal T_D^{\mathrm{DM}}
$$
by forgetting the Glynn decoration $\delta$, the distinguished framed
vertex $v$, the marked source leg $\ell_0$, the indicator
$\epsilon$, and the rigidifying boundary coordinates, and retaining
only the underlying rooted metrized polynomial tree with its harmonic
self-map, and then consider the subtree spanned by the Julia set of the induced map on the boudary, and collapse each maximal subtree in its complement into a single point, pushing some marked legs into interior (type II) points. When distinct legs are pushed to the same type II point we still keep them distinguished. Since we have restricted to $\sigma_\Xi^{\mathrm{dyn}}$, the
resulting DeMarco--McMullen tree metric is non-zero.

The maps $\Psi_\Xi$ are compatible with specialization. Indeed, if
$\Xi'$ is obtained from $\Xi$ by a specialization of rigidified
framed decorated types, then the face map
$$
j_{\Xi',\Xi}\colon \sigma_{\Xi'}\hookrightarrow \sigma_\Xi
$$
is induced by contracting or specializing the corresponding decorated
tree data. Forgetting the decorations sends this operation to the
corresponding contraction or specialization of the underlying
DeMarco--McMullen polynomial tree. Hence the diagram
$$
\begin{CD}
\sigma_{\Xi'}^{\mathrm{dyn}} @>{j_{\Xi',\Xi}}>>
\sigma_{\Xi}^{\mathrm{dyn}} \\
@V{\Psi_{\Xi'}}VV @VV{\Psi_\Xi}V \\
\mathcal T_D^{\mathrm{DM}} @= \mathcal T_D^{\mathrm{DM}}
\end{CD}$$
commutes wherever the terms are defined. Therefore the local maps glue
to a continuous map of cone complexes
$$
\Psi_{\vec\mu^*}\colon 
\mathcal T_D(\vec\mu^*)^{\mathrm{dyn}}
\to
\mathcal T_D^{\mathrm{DM}}.
$$

This map is equivariant with respect to the scaling action of
$\mathbb R_{>0}$. Indeed, scaling a decorated tree multiplies all
length coordinates by the same positive constant, and after forgetting
decorations this is exactly the usual metric rescaling of the underlying
DeMarco--McMullen tree. Thus, we have
$
\Psi_{\vec\mu^*}(\lambda\cdot \tau)
=
\lambda\cdot \Psi_{\vec\mu^*}(\tau)
$
for all $\lambda\in\mathbb R_{>0}$. Hence $\Psi_{\vec\mu^*}$ descends
to a continuous map on projectivizations
$$
\Psi_{\vec\mu^*}\colon 
\mathbb P\mathcal T_D(\vec\mu^*)^{\mathrm{dyn}}
\to
\mathbb P\mathcal T_D^{\mathrm{DM}}.
$$

We now identify the image. Let
$
[\tau]\in \mathbb P\mathcal T_D(\vec\mu^*)^{\mathrm{dyn}}.
$
Before forgetting decoration, the finite critical points are labelled by
the profile $\vec\mu^*$. After forgetting the rigidified framed Hurwitz
decoration, some of these critical points may become indistinguishable
on the underlying DeMarco--McMullen tree. Geometrically, this is exactly
the collision of finite critical points. When critical points with
multiplicities
$
\mu_{i_1},\ldots,\mu_{i_s}
$
collide, their multiplicities add and give a single critical point of
multiplicity
$\mu_{i_1}+\cdots+\mu_{i_s}.
$
Thus the resulting DeMarco--McMullen ramification profile is a coarser
profile
$
\vec\mu'^*\succeq \vec\mu^*.
$
It follows that
$$
\Psi_{\vec\mu^*}
\bigl(
\mathbb P\mathcal T_D(\vec\mu^*)^{\mathrm{dyn}}
\bigr)
\subset
\bigcup_{\vec\mu'^*\succeq \vec\mu^*}
\mathbb P\mathcal T_D^{\mathrm{DM}}(\vec\mu'^*).
$$

On the open non-collision locus, no finite critical points collide after
forgetting the decoration. Therefore the underlying DeMarco--McMullen
tree has the same ramification profile $\vec\mu^*$, and hence the image
of this locus is contained in
$
\mathbb P\mathcal T_D^{\mathrm{DM,adm}}(\vec\mu^*).
$
Conversely, take a point
$
[\tau_{\mathrm{DM}}]\in
\mathbb P\mathcal T_D^{\mathrm{DM},\mathrm{adm}}(\vec\mu^*).
$
By definition, and by the fact that the classical (type 1 point) Julia set is contained in any non trivial closed invariant subset \cite[Corollary 5.24]{benedetto2019dynamics}, this DeMarco--McMullen tree is induced by an
admissible cover degeneration with ramification profile $\vec\mu^*$.
Choosing the corresponding rigidified framed Hurwitz decoration, namely
the Glynn decoration, the distinguished framed vertex $v$, the marked
source leg $\ell_0$, and the value of $\epsilon$, gives a point of
$
\mathbb P\mathcal T_D(\vec\mu^*)^{\mathrm{dyn}}
$
lying in the open non-collision locus. Its image under
$\Psi_{\vec\mu^*}$ is $[\tau_{\mathrm{DM}}]$.
Therefore the image of
the open non-collision locus is exactly
$\mathbb P \mathcal T_D^{\mathrm{DM,adm}}(\vec\mu^*).
$

Finally, the description of the fibers follows from the definition of
the map. Two points of
$\mathbb P \mathcal T_D(\vec\mu^*)^{\mathrm{dyn}}$
have the same image in $\mathbb P \mathcal T_D^{\mathrm{DM}}$ if and only if, after
forgetting the rigidified framed Hurwitz decoration and the purely
rigidifying boundary coordinates, as well as collapsing the connected components of the complement of the subtree spanned by the Julia set at the boundary, their dynamical length data determine
the same projectivized DeMarco--McMullen polynomial tree. Thus the fiber
over a DeMarco--McMullen tree consists, up to the Hurwitz
equivalence, the choices of Glynn decoration, framed target component
$v$, marked source leg $\ell_0$, and indicator $\epsilon$, together
with the possible rigidifying boundary length parameters, which give the same underlying projectivized
polynomial tree.

This proves Theorem \ref{thm:main-relation-to-DM}.

\bibliographystyle{siam}
\bibliography{references}

\end{document}